\documentclass[12pt,a4paper,draft,reqno]{amsart}
\usepackage{amssymb,a4wide,amscd,times,mathptmx}
\numberwithin{equation}{section}
\newtheorem{thm}{Theorem}
\numberwithin{thm}{section}
\newtheorem{prop}[thm]{Proposition}
\newtheorem{lem}[thm]{Lemma}
\newtheorem{cor}[thm]{Corollary}

\theoremstyle{definition}
\newtheorem{defn}[thm]{Definition}
\newtheorem{exmp}[thm]{Example}
\theoremstyle{remark}
\newtheorem{rem}[thm]{Remark}
\providecommand{\BBb}[1]{{\mathbb{#1}}}

\providecommand{\cal}[1]{{\mathcal{#1}}}   

\newcommand{\C}{{\BBb C}}
\newcommand{\dual}[2]{\langle\,#1,\,#2\,\rangle}
\newcommand{\im}{\operatorname{i}}
\newcommand{\lap}{\operatorname{\Delta}}
\newcommand{\lOm}{\ell_{\Omega}}

\newcommand{\mlap}{-\!\operatorname{\Delta}}
\newcommand{\nrm}[2]{\|#1\|_{#2}}
\newcommand{\Nrm}[2]{\bigl\|#1\bigr\|_{#2}}
\newcommand{\order}{\operatorname{order}}
\newcommand{\op}[1]{\operatorname{#1}}
\newcommand{\OP}{\operatorname{OP}}
\newcommand{\N}{\BBb N}
\newcommand{\R}{{\BBb R}}
\newcommand{\Rn}{{\BBb R}^{n}}
\newcommand{\rOm}{r_{\Omega}}
\newcommand{\scal}[2]{(\,#1\,|\, #2\,)}
\newcommand{\set}[2]{\{\,#1 \mid #2\,\}}
\newcommand{\Set}[2]{\bigl\{\,#1\bigm| #2\,\bigr\}}
\newcommand{\singsupp}{\operatorname{sing\,supp}}
\newcommand{\supp}{\operatorname{supp}}
\newcommand{\Z}{\BBb Z}
\renewcommand{\hat}[1]{\overset{{\scriptscriptstyle \wedge}}{#1}}
\newcommand{\Bbar}{\overline{B}}
\newcounter{rmcount}\renewcommand{\thermcount}{{\rm\roman{rmcount}}}
\newenvironment{rmlist}{%
\begin{list}{{\rm(\thermcount)}}{\setlength{\labelwidth}{\leftmargin}%
\usecounter{rmcount}}}{\end{list}}
\begin{document}
\title[Fundamental Results for Operators of Type $\mathbf{1},\mathbf{1}$]%
{Fundamental Results for\\Pseudo-Differential Operators of Type $\mathbf{1},\mathbf{1}$}
\author{Jon Johnsen}
\address{Department of Mathematical Sciences\\ Aalborg University\\ 
Fredrik Bajers Vej 7G\\ DK-9220 Aalborg {\O}st, Denmark}
\email{jjohnsen@math.aau.dk}
\thanks{Supported by the Danish Council for Independent Research, Natural Sciences
  (Grant No.~4181-00042)
\\[3\baselineskip]
{\tt Appeared in Axioms, vol. 5 (2016), no. 2, 13. doi:10.3390/axioms5020013}} 

\enlargethispage{3\baselineskip}
\begin{abstract}
This paper develops some deeper consequences of an extended definition, proposed previously by the author, 
of pseudo-differential operators that are of type $1,1$
in H{\"o}rmander's sense. Thus, it contributes to the long-standing problem of creating a
systematic theory of such operators. 
It is shown that type $1,1$-operators are 
defined and continuous on the full space of temperate distributions, if they fulfil H{\"o}rmander's twisted
diagonal condition, or more generally if they belong to the self-adjoint subclass;
and that they are always defined on the temperate smooth functions.  As~a~main tool the
paradifferential decomposition is derived for type $1,1$-operators,
and to confirm a natural hypothesis the symmetric term is shown to cause the domain restrictions;
whereas the other terms are shown to define nice type $1,1$-operators fulfilling the twisted diagonal condition.
The decomposition is analysed in the type $1,1$-context by combining the Spectral Support Rule
and the factorisation inequality, which gives pointwise estimates of pseudo-differential operators in terms of maximal functions. 
\end{abstract}
\maketitle

\section{Introduction}  \label{intro-sect}
Pseudo-differential operators $a(x,D)$ of type $1,1$ have long been known to have peculiar
properties, almost since their invention by H{\"o}rmander~\cite{Hrd}. 
This is due to initial investigations in 1972 in the
thesis of Ching \cite{Chi72} and in lecture notes of Stein (made available in \cite[Ch.~VII\S1.3]{Ste93}); 
and again in 1978 by Parenti and Rodino \cite{PaRo78}.

The understanding of their unusual theory, and of the applications of these linear operators to
non-linear problems in partial differential equations,
grew crucially in the 1980's through works of Meyer \cite{Mey80,Mey81}, Bony \cite{Bon}, 
Bourdaud \cite{Bou82,Bou83,Bou87,Bou88}, H{\"o}rmander~\cite{H88,H89}.
Cf.\ also the expositions of H{\"o}rmander~\cite[Ch.~9]{H97} and Taylor~\cite{Tay91}.

However, the first formal definition of general type $1,1$-operators was put forward 
in 2008 by the author~\cite{JJ08vfm}. It would not be unjust to view this as an axiomatization of
the type $1,1$-theory, for whereas the previous contributions did not attempt to crystallise what a
type $1,1$-operator \emph{is} or how it can be \emph{characterised} in general, the definition from \cite{JJ08vfm} has been a 
fruitful framework for raising questions and seeking answers about type $1,1$-operators.

Indeed, being based on an operator theoretical approach, mimicking unbounded operators in Hilbert
space, the definition gave from the outset a rigorous discussion of, e.g.,\ unclosability, pseudo-locality,
non-preservation of wavefront sets and the Spectral Support Rule \cite{JJ08vfm}. 
This was followed up with a systematic $L_p$-theory of type $1,1$-operators in \cite{JJ11lp}, where a
main theorem relied on a symbol analysis proved in full detail in the present paper.

Meanwhile, M{\'e}tivier also treated type $1,1$-operators in 2008 in Chapter~4 of
\cite{Met08}, but took recourse to the space dependent extensions of 
Stein~\cite[Ch.~VII\S1.3]{Ste93}. 
Type $1,1$-operators have also been investigated, or played a role, in works of e.g.\ 
Torres~\cite{Tor90},
Marschall~\cite{Mar91},
Grafakos and Torres~\cite{GrTo99}, 
Taylor~\cite{Tay00},
H\'erau~\cite{Her02},
Lannes~\cite{Lan06},
Johnsen~\cite{JJ08par},
Hounie and dos Santos Kapp~\cite{HoSK09};
and for bilinear operators in Bernicot and Torres~\cite{BeTo11}.
Implicitly type $1,1$-operators also enter many works treating partial
differential equations with Bony's paradifferential calculus; but this
would lead too far to recall~here.

The present paper goes into a deeper, 
systematic study of type $1,1$-operators on $\cal S'(\Rn)$ and its subspaces.
Indeed, the definition in \cite{JJ08vfm} is shown here to give operators
always defined on the maximal smooth subspace $C^\infty\cap\cal S'$,
generalising results of Bourdaud~\cite{Bou87}  and David and Journ\'e~\cite{DaJo84}\,---\,and shown to be defined on the entire $\cal S'$ 
if they belong to the self-adjoint subclass, by an extension of 
H{\"o}rmander's analysis of this class~\cite{H88,H89}.
Moreover, the pointwise estimates in \cite{JJ11pe} are applied to the para\-differential
decompositions, which are analysed in the type $1,1$-context here. The decomposition gives 3 other
type~$1,1$-operators, of which the so-called symmetric term is responsible for the possible domain restrictions,
which occur when its infinite series diverges.

Altogether this should bring the theory of type $1,1$-operators to a rather more mature level.

\subsection{Background}
Recall that the symbol $a(x,\eta)$ of a type $1,1$-operator of order
$d\in \R$ fulfils
\begin{equation}
  |D^\alpha_\eta D^\beta_x a(x,\eta)|\le
  C_{\alpha,\beta}(1+|\eta|)^{d-|\alpha|+|\beta|}\quad\text{for}\quad
  x,\eta\in \Rn.  
  \label{Cab-ineq}
\end{equation}
Classical pseudo-differential operators are, e.g.,\ partial differential operators $\sum_{|\alpha|\le d}
a_\alpha(x)D^\beta$, having such symbols simply with $d-|\alpha|$ as exponents. The presence of
$|\beta|$ allows for a higher growth with respect to $\eta$, which has attracted attention for a
number of reasons.

The operator corresponding to \eqref{Cab-ineq} is for Schwartz functions $u(x)$, i.e.\ $u\in \cal S(\Rn)$,
\begin{equation}
   a(x,D)u=(2\pi)^{-n}\int e^{\im x\cdot \eta}a(x,\eta)\hat u(\eta)\,d\eta.  
\label{aSS-eq}
\end{equation}
But for $u\in \cal S'\setminus \cal S$ it requires another definition to settle whether $u$ belongs 
to the domain of $a(x,D)$ or not. This is indeed a main subject of the present paper,
which exploits the general definition of $a(x,D)$ presented in \cite{JJ08vfm};
it is recalled in \eqref{a11-id} below.

The non-triviality of the above task was discovered already by Ching \cite{Chi72},
who showed unboundedness on $L_2$ for certain $a_\theta(x,D)$ with $d=0$; cf.\
Example~\ref{Ching-exmp} below. As the adjoint $a_\theta(x,D)^*$ of Ching's operator does not leave $\cal S$ invariant,
as can be seen explicitly e.g.\ from the proof of \cite[Lem.~3.1]{JJ08vfm},
the usual extension to $\cal S'$ by duality is not possible for $\OP(S^d_{1,1})$.

In general the pathologies of type $1,1$-operators are without doubt reflecting 
that, most interestingly, this operator class has important applications to non-linear problems:

This was first described around 1980 by Meyer \cite{Mey80,Mey81}, 
who discovered that a composition operator $u\mapsto F\circ u=F(u)$ with $F\in
C^\infty $, $F(0)=0$,  
can be decomposed in its action on functions $u\in \bigcup _{s>n/p} H^s_p(\Rn)$,
by means of a specific $u$-dependent type $1,1$ symbol $a_u(x,\eta)\in S^0_{1,1}$, as
\begin{equation}
  F(u(x))=a_u(x,D)u(x).
  \label{auxDu-eq}
\end{equation}
He also showed that $a_u(x,D)$ \emph{extends} to a bounded operator on $H^t_r$ for $t>0$, 
so the fact that $u\mapsto F(u)$ sends $H^s_{p}$ into itself
can be seen from \eqref{auxDu-eq} by taking $t=s$ and $r=p$\,---\,indeed, 
this proof method is particularly elegant for non-integer $s>n/p$.
It was carried over rigorously to the present type $1,1$-framework in \cite[Sec.~9]{JJ08vfm}, with
continuity of $u\mapsto F\circ u$ as a corollary.
Some applications of \eqref{auxDu-eq} were explained by Taylor~\cite[Ch.~3]{Tay91}.

Secondly, it was shown in \cite{Mey81} that type $1,1$-operators play a main role in
the paradifferential calculus of Bony \cite{Bon} and the
microlocal inversion of nonlinear partial differential equations of the form
\begin{equation}
  G(x,(D^\alpha_x u(x))_{|\alpha|\le m})=0.
  \label{GxDu-eq}
\end{equation}
This was explicated by H{\"o}rmander, who devoted 
\cite[Ch.~10]{H97} to the subject. The resulting set-up was used e.g.\ by 
H\'erau \cite{Her02} in a study of hypoellipticity of \eqref{GxDu-eq}.
Moreover, it was used for propagation of singularities 
in \cite[Ch.~11]{H97}, with special emphasis on non-linear hyperbolic equations. 
Recently paradifferential operators, and thus type $1,1$-operators,
were also exploited for non-linear Schr{\"o}dinger operators in constructions of solutions, parametrices 
and propagation of singularities in global wave front sets; cf.\ works of e.g.\ 
Tataru~\cite{Tat08}, Delort~\cite{Del11}, Nicola and Rodino~\cite{NiRo15}.

Thirdly, both type $1,1$-theory as such and Bony's paradifferential
techniques played a crucial role in the author's work on 
semi-linear elliptic boundary problems \cite{JJ08par}. 

Because of the relative novelty of this application, a sketch is given using a typical example.
In a bounded $C^\infty $-region $\Omega\subset \Rn$ with normal derivatives 
$\gamma_ju= (\vec n\cdot \nabla)^j u$ at the boundary $\partial\Omega$,
and $\lap:=\partial^2_{x_1}+\dots+\partial^2_{x_n}$,
let $u(x)$ solve the perturbed $\ell$-harmonic Dirichl\'et problem 
\begin{equation} 
  (\mlap)^\ell u+u^2=f \quad\text{ in } \Omega,
\qquad
  \gamma_ju=\varphi_j\quad\text{ on } \partial\Omega\text{ for } 0\le j<\ell.
\label{u2-pb}
\end{equation}

Without $u^2$, the linear problem has a well-known solution $u_0=R_\ell f+K_0\varphi_0+\dots
+K_{\ell-1}\varphi_{l-1}$, with operators belonging to the pseudo-differential boundary operator
class of Boutet de Monvel~\cite{BM71}.
For~the non-linear problem in
\eqref{u2-pb}, the parametrix construction of \cite{JJ08par} yields the
solution formula 
\begin{equation}
  u= P^{(N)}_u(R_\ell f +K_0\varphi_0+\dots +K_{\ell-1}\varphi_{\ell-1})
    +(R_\ell L_u)^Nu,
  \label{param-eq}
\end{equation}
where the parametrix $ P^{(N)}_u$ is the linear map given by the finite Neumann series
\begin{equation}
 P^{(N)}_u=I +R_\ell L_u+\dots +(R_\ell L_u)^{N-1}  
\label{param'-eq}
\end{equation}
in terms of the \emph{exact} paralinearisation $L_u$ of $u^2$ with the sign convention $-L_u(u)=u^2$;
cf.\ \cite{JJ08par}. 

One merit of \eqref{param-eq} is to show why $u$'s regularity is \emph{unchanged} by the non-linear term $u^2$: 
each parametrix $P^{(N)}_u$ is of order $0$, hence does not change Sobolev regularity when applied to $u_0$;
while in \eqref{param-eq} the remainder $(R_\ell L_u)^{N}u$ will be
in $C^k(\overline{\Omega})$ for every fixed $k$ if $N$ is taken large 
enough. Indeed, $R_\ell L_u$ has a fixed negative order if $u$
is given with just the weak a priori regularity necessary to make sense of the
boundary condition and make $u^2$ defined and a priori more regular than
$(\mlap)^\ell u$.

Type $1,1$-operators are important for the fact that \eqref{param-eq} easily implies 
that extra regularity properties of $f$ in subregions $\Xi\Subset \Omega$ carry over to $u$; 
e.g.\ if $f|_{\Xi}$ is $C^\infty $ so is $u|_{\Xi}$. 
Indeed, such implications boil down to the fact that the exact paralinearisation $L_u$ \emph{factors} through an
operator $A_u$ of type $1,1$, that is, if 
$\rOm$ denotes restriction to $\Omega$  and $\lOm$ is a linear extension to $\Rn\setminus\Omega$, 
\begin{equation}
  L_u=\rOm A_u \lOm,\qquad A_u\in \OP(S^\infty _{1,1}).
  \label{rAul-eq}
\end{equation}
Now, by inserting \eqref{rAul-eq} into \eqref{param'-eq}--\eqref{param-eq} 
for a large $N$ and using cut-off functions
supported in $\Xi$ in a well-known way, cf.\ \cite[Thm.~7.8]{JJ08par},
the regularity of $u$ locally in $\Xi$ is at once improved to the extent permitted by the data $f$
by using the \emph{pseudo-local} property of $A_u$:
\begin{equation}
  \singsupp Au\subset \singsupp u\quad\text{for}\quad u\in D(A).
\end{equation}

However, the pseudo-local property of general type $1,1$-operators was only
proved recently in \cite{JJ08vfm}, inspired by the application below \eqref{rAul-eq}. 
Yet, pseudo-locality was anticipated more than three decades ago by Parenti and Rodino \cite{PaRo78}, 
who gave an inspiring but incomplete indication, as they relied on the future to bring a specific
meaning to $a(x,D)u$  for $u\in \cal S'\setminus C^\infty_0$.

A rigorous definition of general type $1,1$ operators was first given in \cite{JJ08vfm}.
In a way the definition abandons Fourier analysis (temporarily) and mimicks the theory of unbounded
operators in Hilbert spaces. This is by viewing a type $1,1$-operator as a densely defined, unbounded
operator $a(x,D)\colon \cal S'\to\cal D'$ between the two topological vector spaces $\cal S'(\Rn)$
and $\cal D'(\Rn)$; thus the graph of $a(x,D)$ may be closed or unclosed in $\cal S'\times\cal D'$ etc.
Indeed, it was proposed in \cite{JJ08vfm} to stipulate that
$u\in\cal S'$ belongs to the domain $D(a(x,D))$ of $a(x,D)$ and to set
\begin{equation}
  a(x,D)u:= \lim_{m\to\infty } (2\pi)^{-n}\int_{\Rn}
  e^{\im x\cdot\eta}\psi(2^{-m}D_x)a(x,\eta)\psi(2^{-m}\eta)\hat u(\eta)\,d\eta
  \label{a11-id}
\end{equation}
whenever this limit does exist in $\cal D'(\Rn)$
for every  $\psi\in C^\infty_0(\Rn)$ with $\psi=1$ in a neighbourhood of the origin, 
and does not depend on such $\psi$. 

In passing it is noted that, beyond the definition, operator theory is also felt in the rules
of calculus, since as shown in Proposition~\ref{comm-prop} below the well-known commutator identity
is replaced for type $1,1$-operators by an operator theoretical \emph{inclusion},
\begin{equation}
  a(x,D)D_j+[D_j,a(x,D)]\subset D_j a(x,D).
\end{equation}

The unconventional definition in \eqref{a11-id}, by \emph{vanishing frequency modulation}, 
is a rewriting of the usual one, which is suitable for the present general symbols:
clearly \eqref{a11-id} gives back the integral in
\eqref{aSS-eq} if $u\in \cal S$. In case $a\in S^d_{1,0}$ this
identification extends further to $u\in \cal S'$ by duality and the calculus of classical
pseudo-differential operators. Note that the above integral should be interpreted as the operator 
$\OP(\psi(2^{-m}D_x)a(x,\eta)\psi(2^{-m}\eta))$ in $\OP(S^{-\infty})$ acting on $u$.

Clearly \eqref{a11-id} is reminiscent of oscillatory integrals, now with the addition
that $u\in D(a(x,D))$ when the regularisation yields a limit independent of
the integration factor. Of course it is not a conventional integration factor that is used here, 
but rather the Fourier multiplier $\psi(2^{-m}D_x)$ that modifies the frequencies of $a(\cdot ,\eta)$.
While the necessity of this modification was amply elucidated in \cite{JJ08vfm}, it is moreover 
beneficial because the use of $\psi(2^{-m}D_x)$ gives easy access to Littlewood--Paley analysis of $a(x,D)$.

The definition \eqref{a11-id} was investigated in \cite{JJ08vfm} from
several other perspectives, of which some will be needed below.
But mentioned briefly \eqref{a11-id} was proved to be maximal
among the definitions of $A=a(x,D)$ that gives back the usual operators in
$\OP(S^{-\infty })$ and is stable under the limit in \eqref{a11-id}; 
$A$ is always defined on $\cal F^{-1}\cal E'$; it is pseudo-local but
does change wavefront sets in certain cases (even if $A$ is defined on $\bigcup H^s$); 
and $A$ transports supports via the distribution kernel, 
i.e.\ $\supp Au\subset \supp K\circ \supp u$ when $u\in D(A)\bigcap\cal E'$,
with a similar \emph{spectral} support rule for $\supp\hat u$; cf.\
\eqref{sFAu-eq} below and Appendix~\ref{spectral-app}.

For the Weyl calculus, H{\"o}rmander~\cite{H88}  
noted that type $1,1$-operators do not fit well, as
Ching's operator can have discontinuous Weyl-symbol. Conversely 
Boulkhemair \cite{Blk95,Blk99} showed that the Weyl operator
$\iint e^{\im(x-y)\cdot \eta}a(\tfrac{x+y}{2},\eta)u(y)\,dy\,d\eta/(2\pi)^n$ 
may give peculiar properties by insertion of $a(x,\eta)$ from $S^d_{1,1}$.
E.g., already for Ching's symbol with $d=0$, 
the real or imaginary part gives a Weyl operator that is unbounded on $H^s$ for \emph{every} $s\in\R$.

For more remarks on the subject's historic development the reader may refer to Section~\ref{prel-sect};
or consult the review in the introduction of~\cite{JJ08vfm}.

\subsection{Outline of Results}
The purpose of this paper is to continue the foundational study in \cite{JJ08vfm}
and support the definition in \eqref{a11-id} with further consequences.

First of all this concerns the hitherto untreated question: 
under which conditions is a given type $1,1$-operator $a(x,D)$
an everywhere defined and continuous map
\begin{equation}
  a(x,D)\colon \cal S'(\Rn)\to\cal S'(\Rn)\quad ?
  \label{S'S'-eq}
\end{equation}
For this it is shown in Proposition~\ref{GB-prop} and Theorem~\ref{tdc-thm} below to be
sufficient that $a(x,\eta)$ fulfils 
H{\"o}rmander's \emph{twisted diagonal condition}, i.e.\ the partially Fourier
transformed symbol 
\begin{equation}
  \hat a(\xi,\eta)=\cal F_{x\to\xi}a(x,\eta)
\end{equation}
should vanish in a conical neighbourhood of a non-compact part of the twisted
diagonal $\mathcal{T}$ given by $\xi+\eta=0$ in $\Rn\times \Rn$.
More precisely this means that for some $B\ge 1$
\begin{equation}
  \hat a(\xi,\eta)\ne 0 \quad\text{only if}\quad  |\xi+\eta|+1\ge |\eta|/B.
  \label{tdc-eq}
\end{equation}
It should perhaps be noted that it is natural to consider $\hat a(\xi,\eta)$,
as it is related (cf.\ \cite[Prop.~4.2]{JJ08vfm}) 
both to the kernel $K$ of $a(x,D)$ and to the kernel
$\cal K$ of $\cal F^{-1}a(x,D)\cal F$,
\begin{equation}
  (2\pi)^{n}\cal K(\xi,\eta)=\hat a(\xi-\eta,\eta)= 
  \cal F_{(x,y)\to(\xi,\eta)}K(x,-y).
  \label{KaFK-eq}
\end{equation}

More generally the $\cal S'$-continuity \eqref{S'S'-eq} is obtained in Theorems~\ref{tildeS-thm} and \ref{sigma-thm} {below} 
for the $a(x,\eta)$ in $S^d_{1,1}$ that merely satisfy H{\"o}rmander's twisted diagonal
condition of \emph{order} $\sigma$ for all $\sigma\in \R$. These are the symbols which
for some $c_{\alpha,\sigma}$ and $0<\varepsilon<1$ fulfil
\begin{equation}
  \sup_{x\in \Rn,\; R>0} R^{|\alpha|-d}\big(
  \int_{R\le |\eta|\le 2R} |D^\alpha_\eta a_{\chi,\varepsilon}(x,\eta)|^2
  \,\frac{d\eta}{R^n}\big)^{1/2}
  \le c_{\alpha,\sigma} \varepsilon^{\sigma+n/2-|\alpha|}.
  \label{tdcs-eq}
\end{equation}
In this asymptotic formula $\hat a_{\chi,\varepsilon}$ denotes a specific
localisation of $\hat a(x,\eta)$ to the conical neighbourhood
$|\xi+\eta|+1\le 2\varepsilon|\eta|$ of the twisted diagonal $\mathcal{T}$.

Details on the cut-off function $\chi$ in \eqref{tdcs-eq} are recalled in
Section~\ref{TDC-ssect}, in connection with an account of 
H{\"o}rmander's fundamental result that validity of \eqref{tdcs-eq} for all
$\sigma\in\R$ is equivalent to extendability 
of $a(x,D)$ to a bounded map $H^{s+d}\to H^s$ for all $s\in\R$, as
well as equivalent to the adjoint $a(x,D)^*$ being of type $1,1$.

Of course these results of H{\"o}rmander make it natural to expect that
the above two conditions (namely \eqref{tdc-eq} and \eqref{tdcs-eq}
for all $\sigma$) are sufficient for the $\cal S'$-continuity in \eqref{S'S'-eq}, but this has not been addressed 
explicitly in the literature before.
As mentioned they are verified in Theorem~\ref{tdc-thm}, respectively in Theorem~\ref{tildeS-thm} by
duality and in Theorem~\ref{sigma-thm} by exploiting \eqref{tdcs-eq} directly.

In the realm of smooth functions the situation is fundamentally different. Here  there is
a commutative diagram for \emph{every} type $1,1$-operator $a(x,D)$:
\begin{equation}
\begin{CD}
  \cal S  @>>>  \cal S+\cal F^{-1}\cal E' @>>> \cal O_M  @>>> 
  C^\infty\bigcap \cal S'
\\
  @V a(x,D) VV  @V a(x,D) VV  @VV a(x,D) V  @VV a(x,D) V
\\
  \cal S  @>>>  \cal O_M @>>> \cal O_M  @>>>
  C^\infty
  \end{CD}
\label{aCD-eq}
  \end{equation}
The first column is just the integral \eqref{aSS-eq}; the second an
extension from \cite{JJ05DTL,JJ08vfm}. Column three is an improvement given below of 
the early contribution of Bourdaud \cite{Bou87} that
$a(x,D)$ extends to a map $\cal O_M\to \cal D'$,
whereby $\cal O_M$ denotes Schwartz' space of slowly increasing smooth functions.

However, the fourth column restates the full result
that a type $1,1$-operator is always defined on the 
\emph{maximal} space of smooth functions $C^\infty\bigcap\cal S'$.
More precisely, according to Theorem~\ref{aOO-thm} below, it restricts to a 
strongly continuous map
\begin{equation}
  a(x,D)\colon C^\infty (\Rn)\bigcap\cal S'(\Rn)\to C^\infty (\Rn).
  \label{aCSC-eq}
\end{equation}
It is noteworthy that this holds without any of the conditions
\eqref{tdc-eq} and \eqref{tdcs-eq}. 
Another point is that, since $C^\infty\not\subset\cal S'$, it was necessary to ask for a limit in
the topology of $\cal D'$ in \eqref{a11-id}.

Perhaps it could seem surprising that the described results on \eqref{S'S'-eq} and 
\eqref{aCSC-eq} have not been established in their full generality before. However, it should
be emphasised that these properties are valid for the 
operator defined in \eqref{a11-id}, so they go much beyond the
mere extendability discussed by Meyer~\cite{Mey81}, Bourdaud~\cite{Bou88},
H{\"o}rmander~\cite{H88,H89,H97}, Torres~\cite{Tor90}, Stein~\cite{Ste93}.

The definition in \eqref{a11-id} is also useful because it easily adapts to
Littlewood--Paley analysis of type $1,1$-operators. Here the systematic point of departure
is the well-known paradifferential splitting based on dyadic coronas
(cf.\ Section~\ref{corona-sect} for details),
as used by e.g.\ Bony \cite{Bon}, Yamazaki~\cite{Y1}, Marschall~\cite{Mar91}:
\begin{equation}
  a(x,D)=a^{(1)}_\psi(x,D)+a^{(2)}_\psi(x,D)+a^{(3)}_\psi(x,D).
  \label{parasplit-eq}
\end{equation}
Since the 1980's splittings like \eqref{parasplit-eq} have been used
in microlocal analysis of \eqref{GxDu-eq} as well as in
numerous proofs of continuity of $a(x,D)$ in Sobolev spaces $H^s_p$ and
H{\"o}lder--Zygmund spaces $C^s_*$, or the more general Besov and
Lizorkin--Triebel scales $B^{s}_{p,q}$ and $F^{s}_{p,q}$. 
For type $1,1$ operators \eqref{parasplit-eq} was used by
Bourdaud~\cite{Bou82,Bou83,Bou88}, Marschall~\cite{Mar91}, Runst~\cite{Run85ex},
and the author in \cite{JJ08vfm,JJ11lp},
and in \cite{JJ04DCR,JJ05DTL} where the Lizorkin--Triebel spaces 
$F^s_{p,1}$ were shown to be optimal substitutes for the Sobolev spaces $H^s_p$ 
at the borderline $s=d$ for the domains of operators in $\OP(S^d_{1,1})$.

It is known that the decomposition \eqref{parasplit-eq} follows from 
the bilinear way $\psi$ enters \eqref{a11-id}, and that one finds at once the three
infinite series in \eqref{a1-eq}--\eqref{a3-eq} below, which define the $a^{(j)}_\psi(x,D)$.
But it is a main point of Sections~\ref{corona-sect} and \ref{split-sect} to verify 
that \emph{each} of these series gives an operator $a^{(j)}_\psi(x,D)$ also belonging to
$\OP(S^d_{1,1})$; which is non-trivial because of the modulation function $\psi$ in \eqref{a11-id}.

As general properties of the type $1,1$-operators $a^{(1)}_\psi(x,D)$ and $a^{(3)}_\psi(x,D)$,
they are shown here to satisfy the twisted diagonal condition \eqref{tdc-eq}, so
\eqref{parasplit-eq} can be seen as a main source of such operators. 
Consequently these terms are harmless as they are defined on $\cal S'$ because of \eqref{S'S'-eq} ff.

Therefore, it is the so-called symmetric term  $a^{(2)}_{\psi}(x,D)$ which may
cause $a(x,D)u$ to be undefined, as was previously known e.g.\ for functions $u$
in a Sobolev space; cf.\ \cite{JJ05DTL}. This delicate situation is
clarified in Theorem~\ref{a2a-thm} with a natural identification of type $1,1$-domains, namely
\begin{equation}
  D(a(x,D))=D(a^{(2)}_\psi(x,D)).
\label{dom-rel}
\end{equation}
This might seem obvious at first glance, but really is without meaning before the
$a^{(2)}_\psi$-series has been shown to define a type $1,1$-operator. Hence \eqref{dom-rel} is 
a corollary to the cumbersome book-keeping needed for this identification of $a_\psi^{(2)}(x,D)$. In fact, the real meaning of
\eqref{dom-rel} is that both domains consist of the $u\in \cal S'$ for which
the $a_\psi^{(2)}$-series converges; cf.\ Theorem~\ref{a2a-thm}.

In comparison, convergence of the series for $a^{(1)}_\psi(x,D)u$ and 
$a^{(3)}_\psi(x,D)u$ is  in Theorem~\ref{a123-thm} verified explicitly for all $u\in \cal S'$,
$a\in S^\infty_{1,1}$, and these operators are proved to be of type $1,1$.
Thus \eqref{parasplit-eq} is an identity among type $1,1$-operators. 
It was exploited for estimates of arbitrary $a\in S^d_{1,1}$ in e.g.\ 
Sobolev spaces $H^s_p$ and H{\"o}lder--Zygmund spaces $C^s_*$ in \cite{JJ11lp}, 
by giving full proofs (i.e.\ the first based on \eqref{a11-id}) 
of the boundedness for all $s>0$, $1<p<\infty$,
\begin{equation}
  a(x,D)\colon H^{s+d}_p(\Rn)\to H^s_p(\Rn),\qquad
  a(x,D)\colon C^{s+d}_*(\Rn)\to C^s_*(\Rn).
  \label{HCs-eq}
\end{equation}
This was generalised in \cite{JJ11lp} to all $s\in\R$ when $a$ fulfills
the twisted diagonal condition of order $\sigma$ in  \eqref{tdcs-eq} 
for all $\sigma\in \R$. This sufficient condition extends results for $p=2$ of
H{\"o}rmander \cite{H88,H89} to $1<p<\infty$, $s\in\R$. The special case $s=0=d$ was considered
recently in \cite{HoSK09}.

The present results on $a_\psi^{(j)}(x,D)$ are of course natural, but they do rely on two techniques
introduced rather recently in works of the author. One ingredient is a \emph{pointwise} estimate
\begin{equation}
  |a(x,D)u(x)|\le c\,u^*(x),\qquad x\in \Rn,
\end{equation}
cf.\ Section~\ref{pe-sect} and \cite{JJ11pe}, 
in terms of the Peetre--Fefferman--Stein maximal function 
\begin{equation}
  u^*(x)=\sup_{y\in\Rn}\frac{|u(x-y)|}{(1+R|y|)^N}, 
  \quad\text{when}\quad \supp\hat u\subset \overline{B}(0,R).
\end{equation}
Although $u\mapsto u^*$ is non-linear, it is useful
for convergence of series: e.g.\ in $H^s_p$ since it is $L_p$-bounded, and as shown here
also in $\cal S'$ since it has polynomial bounds $u^*(x)\le c(1+R|x|)^N$.

The second ingredient is the \emph{Spectral Support Rule} from \cite{JJ08vfm}; cf.\ also
\cite{JJ04DCR,JJ05DTL}. It provides control of
$\supp\cal F(a(x,D)u)$ in terms of the supports of $\hat u$ and 
$\cal K(\xi,\eta)$ in \eqref{KaFK-eq}, 
\begin{equation}
  \supp \cal F(a(x,D)u)\subset \overline{\supp\cal K\circ \supp\cal F u}
  = \bigl\{\,\xi+\eta \bigm| (\xi,\eta)\in \supp \hat a,\
                                 \eta\in \supp\hat u\,\bigr\}^{\overline{\hspace{1ex}}}.
  \label{sFAu-eq}
\end{equation}
The simple case in which $u\in\cal S$ was covered by Metivier
\cite[Prop.~4.2.8]{Met08}. A review of \eqref{sFAu-eq} is given in Appendix~\ref{spectral-app},
including an equally easy proof for arbitrary $\hat u\in \cal E'$ and $a\in S^d_{1,1}$.

A main purpose of \eqref{sFAu-eq} is to avoid a cumbersome approximation by elementary symbols.
These were introduced by Coifman and Meyer~\cite{CoMe78} to reduce the task of bounding the support
of $\cal F(a(x,D)u)$: indeed, elementary symbols have the form 
$a(x,\eta)=\sum m_j(x)\Phi_j(\eta)$ for multipliers $m_j\in L_\infty $ and a Littlewood--Paley
partition $1=\sum \Phi_j$, so clearly
$(2\pi)^{n}\cal Fa(x,D)u=\sum \hat m_j*(\Phi_j\hat u) $ is a finite sum when $\hat u\in \cal E'$;
whence the rule for convolutions yields \eqref{sFAu-eq} for such symbols.

However, approximation by elementary symbols is not just technically redundant because of
\eqref{sFAu-eq}, it would also be particularly cumbersome to use 
for a type $1,1$-symbol, as \eqref{a11-id} would then have to be replaced by a double-limit procedure. 
Moreover, in the proof of \eqref{parasplit-eq}, as well as in the $L_p$-theory based on it in \cite{JJ11lp}, 
\eqref{sFAu-eq} also yields a significant simplification. 

\begin{rem}
The Spectral Support Rule \eqref{sFAu-eq} shows clearly that H{\"o}rmander's
twisted diagonal condition \eqref{tdc-eq} ensures that $a(x,D)$ cannot change
(large) frequencies in $\supp\hat u$ to $0$. In fact, the support condition in \eqref{tdc-eq} 
implies that $\xi$ cannot be close to $-\eta$  when $(\xi,\eta)\in \supp\hat a$,
which by \eqref{sFAu-eq} means that $\eta\in \supp\hat u$ will be changed
by $a(x,D)$ to the frequency $\xi+\eta\ne 0$.  
\end{rem}

\subsection*{Contents}
Notation is settled in Section~\ref{prel-sect} along with basics on
operators of type $1,1$ and the $C^\infty$-results in \eqref{aCD-eq} ff. In Section~\ref{pe-sect}
some pointwise estimates are recalled from \cite{JJ11pe} and then extended to
a version for frequency modulated operators. 
Section~\ref{adj-sect} gives a precise
analysis of the self-adjoint part of $S^d_{1,1}$, relying on the 
results and methods from H{\"o}rmander's lecture notes \cite[Ch.~9]{H97};
with consequences derived from the present operator definition.
Littlewood--Paley analysis of type $1,1$-operators is developed in Section~\ref{corona-sect}. 
In Section~\ref{split-sect}
the operators resulting from the paradifferential splitting
\eqref{parasplit-eq} is further analysed, especially concerning their continuity on
$\cal S'(\Rn)$ and the domain relation \eqref{dom-rel}. 
Section~\ref{final-sect} contains a few final remarks.

\section{Preliminaries on Type $1,1$-Operators}   \label{prel-sect}
Notation and notions from Schwartz' distribution theory, such as the spaces
$C^\infty_0$, $\cal S$, $C^\infty$ of smooth functions and
their duals $\cal D'$, $\cal S'$, $\cal E'$ of distributions,
and the Fourier transformation ${\cal F}$, will be as in H{\"o}rmander's
book \cite{H} with these exceptions:
$\dual{u}{\varphi}$ denotes the value of a distribution $u$ on a test
function $\varphi$. The Sobolev space of order $s\in \mathbb{R}$ based on $L_p$ is written $H^s_p$, and $H^s=H^s_2$.
The space ${\cal O}_M({\mathbb{R}}^n)$ consists of the slowly increasing $f\in
C^\infty({\mathbb{R}}^n)$, i.e.\ the $f$ that for each multiindex $\alpha$ and
some $N>0$ fulfils $|D^\alpha f(x)|\le c(1+|x|)^{N}$.

As usual $t_{+}=\max(0,t)$ is the positive part of $t\in \mathbb{R}$ whilst
$[t]$ denotes the greatest integer $\le t$. In general, $c$ will denote
positive constants, specific to the place of occurrence.

\subsection{The General Definition of Type $1,1$-Operators}
For type $1,1$-operators the reader may consult \cite{JJ08vfm} for an overview of previous results. 
The present paper is partly a continuation of
\cite{JJ04DCR,JJ05DTL,JJ08vfm}, but it suffices to recall just a few facts. 

By standard quantization,
each operator $a(x,D)$ is defined on the Schwartz space $\cal S(\Rn)$ by
\begin{equation}
  a(x,D)u=\OP(a)u(x)
  =(2\pi)^{-n}\int e^{\im x\cdot \eta} a(x,\eta)\cal Fu(\eta)\,d\eta,
 \qquad u\in \cal S(\Rn).
  \label{axDu-id}
\end{equation}
Hereby its symbol $a(x,\eta)$ is required to be in $C^\infty(\Rn\times \Rn)$, of order $d\in \R$
and type $1,1$, which means that for all multiindices $\alpha$, $\beta\in \N_0^n$ it fulfils \eqref{Cab-ineq},
or more precisely has finite seminorms:
\begin{equation}
 p_{\alpha,\beta}(a):= \sup_{x,\eta\in \Rn}
  (1+|\eta|)^{-(d-|\alpha|+|\beta|)}
  |D^\alpha_\eta D^\beta_x a(x,\eta)|
  <\infty.  
  \label{pab-eq}
\end{equation}
The Fr\'echet space of such symbols is denoted by $S^d_{1,1}(\Rn\times \Rn)$, or just $S^d_{1,1}$ for brevity,
while as usual $S^{-\infty}=\bigcap_d S^d_{1,1}$. Basic estimates yield that the bilinear map $(a,u)\mapsto a(x,D)u$ is continuous 
\begin{equation}
  S^d_{1,1}\times \cal S\to \cal S.
  \label{SSS-eq}
\end{equation}
The distribution kernel $K(x,y)=\cal F^{-1}_{\eta\to z}a(x,\eta)\big|_{z=x-y}$ is well known to be
$C^\infty $ for $x\ne y$ also in the type $1,1$ context; cf.\ \cite[Lem.~4.3]{JJ08vfm}. It fulfils 
$\dual{a(x,D)u}{\varphi}=\dual{K}{\varphi\otimes u}$ for all $u$,
$\varphi\in \cal S$.

For arbitrary $u\in \cal S'\setminus \cal S$ it is a delicate
question whether or not $a(x,D)u$ is defined. 
The general definition of type $1,1$-operators in \cite{JJ08vfm} uses a symbol modification,
exploited throughout below, 
namely $b(x,\eta)=\psi(2^{-m}D_x)a(x,\eta)$, or more precisely 
\begin{equation}
  \hat b(\xi,\eta):=\cal F_{x\to\xi}b(x,\eta)=
  \psi(2^{-m}\xi)\hat a(\xi,\eta).
\end{equation}

\begin{defn}   \label{a11-defn}
If a symbol $a(x,\eta)$ is in $S^d_{1,1}(\Rn\times \Rn)$ and $u\in\cal S'(\Rn)$ whilst 
$\psi\in C^\infty_0(\Rn)$ is an arbitrary cut-off function equal to
$1$ in a neighbourhood of the origin, let
\begin{equation}
  a_{\psi}(x,D)u:=
  \lim_{m\to\infty }\op{OP}(\psi(2^{-m}D_x)a(x,\eta)\psi(2^{-m}\eta))u.
  \label{aPsi-eq}
\end{equation}
When for each such $\psi$ the limit 
$a_{\psi}(x,D)u$ exists in $\cal D'(\Rn)$ and moreover is independent of $\psi$, then
$u$ belongs to the domain $D(a(x,D))$ by definition and 
\begin{equation}
  a(x,D)u=a_{\psi}(x,D)u.
  \label{aPsi'-eq}
\end{equation}
This way $a(x,D)$ is a linear map $\cal S'(\Rn)\to\cal D'(\Rn)$ with dense domain,
as by \eqref{SSS-eq} it contains $\cal S(\Rn)$.
(Use of $D(\cdot)$ for the domain should not be confounded with
derivatives, such as $D^\alpha$ or $a(x,D)$.)
\end{defn}

This was called definition by \emph{vanishing frequency modulation} in \cite{JJ08vfm},
because the removal of high frequencies in $x$ and $\eta$ achieved by
$\psi(2^{-m}D_x)$ and $\psi(2^{-m}\eta)$ disappears for $m\to\infty $.
Note that the action on $u$ is well defined for each $m$ in
\eqref{aPsi-eq} as the modified symbol is in $S^{-\infty }$.
Occasionally the function $\psi$ will be referred
to as a \emph{modulation function.}

The frequency modulated operator $\OP(\psi(2^{-m}D_x)a(x,\eta)\psi(2^{-m}\eta))$ has,
by the comparison made in \cite[Prop.~5.11]{JJ08vfm},
its kernel $K_m(x,y)$ conveniently given as a convolution, up to conjugation by the involution
$M\colon (x,y)\mapsto (x,x-y)$, 
\begin{equation} \label{Km-id}
  K_m(x,y)=4^{mn}(\cal F^{-1}\psi(2^m\cdot)\otimes\cal F^{-1}\psi(2^m\cdot))*(K\circ M)(x,x-y).
\end{equation}

\begin{rem} \label{KmS-rem}
It is used below that when $\varphi,\chi\in C_0^\infty(\Rn)$ are such that $\chi\equiv 1$ on
a neighbourhood of $\supp\varphi$, then since  $\supp\varphi\otimes (1-\chi)$ is disjoint from 
the diagonal and bounded in the $x$-direction, 
there is convergence in the topology of $\cal S(\Rn\times \Rn)$:
\begin{equation}
  \varphi(x)(1-\chi(y))K_m(x,y) 
  \xrightarrow[m\to\infty ]{~}\varphi(x)(1-\chi(y)) K(x,y).
\label{KmS-eq}
\end{equation}
However, this requires verification because the commutator of the convolution \eqref{Km-id} and pointwise
multiplication by $\varphi\otimes(1-\chi)$ is a nontrivial pseudo-differential, hence
non-local operator. A proof of \eqref{KmS-eq} based on the Regular Convergence Lemma  
can be found in \cite[Prop.~6.3]{JJ08vfm}.
\end{rem}

In general the calculus of type $1,1$-operators is delicate,
cf.\ \cite{H88,H89,H97}, but the following result from \cite{JJ11lp} is just an exercise (cf.\ the proof there).
It is restated here for convenience.

\begin{prop}
  \label{abc-prop}
When $a(x,\eta)$ is in $ S^{d_1}_{1,1}(\Rn\times \Rn)$ and a symbol with constant coefficients 
$b(\eta)$ belongs to $S^{d_2}_{1,0}(\Rn\times \Rn)$, 
then $c(x,\eta):=a(x,\eta)b(\eta)$
is in $S^{d_1+d_2}_{1,1}(\Rn\times \Rn)$ and
\begin{equation}
  c(x,D)u=a(x,D)b(D)u.
\end{equation}
In particular $D(c(x,D))=D(a(x,D)b(D))$, so the two sides are simultaneously defined.
\end{prop}

This result applies especially to differential operators, say $b(D)=D_{j}$ for simplicity.
But as a minor novelty, the classical commutator identity needs an atypical substitute:

\begin{prop} \label{comm-prop}
  For $a\in S^d_{1,1}$ the commutator 
  \begin{equation}
  [D_j,a(x,D)]= D_j a(x,D)-a(x,D)D_j
  \end{equation}
equals $\OP(D_{x_j}a(x,\eta))$ on the Schwartz space
$\cal S(\Rn)$, whilst in $\cal S'(\Rn)$ there is an inclusion in the operator theoretical sense,
\begin{equation}
  a(x,D)D_j+[D_j,a(x,D)]\subset D_j a(x,D).
\label{dom-incl}
\end{equation}
The commutator symbol $D_{x_j}a(x,\eta)$ is in $S^{d+|\beta|}_{1,1}$.
\end{prop}

\begin{proof}
By classical calculations, any modulation function $\psi $ gives
the following formula for $u\in\cal S$, hence for all $u\in\cal S'$ as the symbols are in $S^{-\infty}$,
  \begin{multline}
    \OP(\psi(2^{-m}D_x)a(x,\eta)\psi(2^{-m}\eta)\eta_j)u+
    \OP(\psi(2^{-m}D_x)D_{x_j}a(x,\eta)\psi(2^{-m}\eta))u
\\
=    D_j \OP(\psi(2^{-m}D_x)a(x,\eta)\psi(2^{-m}\eta))u.
  \end{multline}
When both terms on the left have $\psi$-independent limits for $m\to\infty$, so has the right-hand side. 
As the first term then is $a(x,D)D_j u$, cf.\  Proposition~\ref{abc-prop}, this
entails that the common domain $D(a(x,D)D_j)\bigcap D([D_j,a(x,D)])$ is contained
in that of $D_j a(x,D)$, with the same actions.
\end{proof}

The inclusion \eqref{dom-incl} is strict in some cases, for 
the domains are not always invariant under differentiation. This is a well-known consequence of
the classical counterexamples, which are recalled below for the reader's convenience:

\begin{exmp}
  \label{Ching-exmp}
The classical example of a symbol of type $1,1$ results from an auxiliary
function $A\in C^\infty_0(\Rn)$, say with $\supp A\subset\{\,\eta\mid
\tfrac{3}{4}\le |\eta|\le\tfrac{5}{4}\,\}$, and a fixed vector $\theta\in \Rn$,
\begin{equation}
  a_{\theta}(x,\eta)
=\sum_{j=0}^\infty 2^{jd}e^{-\im 2^j{x}\cdot{\theta}}
  A(2^{-j}\eta).
  \label{ching-eq}
\end{equation}
Here $a_{\theta}\in C^\infty(\Rn\times\Rn)$, since the terms are disjointly
supported, and it clearly belongs to $S^d_{1,1}$.

These symbols were used both by Ching \cite{Chi72} and 
Bourdaud \cite{Bou88} to show $L_2$-unboundedness for $d=0$, $|\theta|=1$.
Refining this, H{\"o}rmander~\cite{H88} established that
continuity $H^s\to\cal D'$ with $s>-r$
holds if and only if $\theta$ is a zero of $A$ of order $r\in\N_0$.
\cite{JJ08vfm} gave an extension to $d\in \R$, $\theta\ne0$.

The non-preservation of wavefront sets discovered by Parenti and
Rodino \cite{PaRo78} also relied on $a_{\theta}(x,\eta)$. 
Their ideas were extended to all $n\ge 1$, $d\in\R$ in \cite[Sect.~3.2]{JJ08vfm} 
and refined by applying $a_{2\theta}(x,D)$ to a product 
$v(x)f(x\cdot \theta)$, where $v\in \cal F^{-1}C^\infty_0$ 
is an analytic function controlling the spectrum,
whilst the highly oscillating $f$ is Weierstrass'
nowhere differentiable function for orders $d\in \,]0,1]$, in a complex
version with its wavefront set along a half-line.
(Nowhere differentiability was shown with a small microlocalisation
argument, explored in \cite{JJ10now}.) 

Moreover, $a_{\theta}(x,D)$ is unclosable in $\cal S'$ when $A$ is 
supported in a small ball around $\theta$, as in \cite[Lem.~3.2]{JJ08vfm}.
Hence Definition~\ref{a11-defn} cannot in general be simplified to
a closure of the graph in $\cal S'\times \cal D'$.
\end{exmp}

As a basic result, it was shown in \cite[Sec.~4]{JJ08vfm}
that the $C^\infty$-subspace $\cal S(\Rn)+\cal F^{-1}\cal E'(\Rn)$ always is contained in
the domain of $a(x,D)$ and that
\begin{equation}
  a(x,D)\colon \cal S(\Rn)+\cal F^{-1}\cal E'(\Rn) \to \cal O_M(\Rn).
  \label{aFE-eq}
\end{equation}
In fact, if $u=v+v'$ is any splitting 
with $v\in \cal S$ and $v'\in \cal F^{-1}\cal E'$, then
\begin{equation}
  a(x,D)u= a(x,D)v+\OP(a(1\otimes \chi))v',
  \label{aFE-id}
\end{equation}
whereby $a(1\otimes \chi)(x,\eta)=a(x,\eta)\chi(\eta)$ and $\chi\in
C^\infty_0(\Rn)$ is arbitrarily chosen so that $\chi=1$ holds in a neighbourhood of 
$\supp\cal Fv'\Subset\Rn$.
Here $a(x,\eta)\chi(\eta)$ is in $S^{-\infty }$ so that 
$\OP(a(1\otimes \chi))$ is defined on $\cal S'$.
Hence $a(x,D)(\cal F^{-1}\cal E')\subset \cal O_M(\Rn)$.

It is a virtue of \eqref{aFE-eq} that $a(x,D)$ is compatible with
for example $\OP(S^\infty_{1,0})$; cf.\ \cite{JJ08vfm} for other compatibility questions. 
Therefore some well-known facts extend to type $1,1$-operators:

\begin{exmp}
  \label{poly-exmp}
Each $a(x,D)$ of type $1,1$ is defined on all polynomials and
\begin{equation}
  a(x,D)(x^\alpha)
  = D^{\alpha}_\eta (e^{\im x\cdot \eta}a(x,\eta))\bigm|_{\eta=0}.
  \label{apoly-eq}
\end{equation}
In fact, $f(x)=x^\alpha$ has
$\hat f(\eta)=(2\pi)^n(-D_\eta)^\alpha\delta_0(\eta)$ with
support $\{0\}$, so it is seen for $v=0$ in \eqref{aFE-id} that
$a(x,D)f(x)=\dual{\hat f}{(2\pi)^{-n}e^{\im\dual{x}{\cdot }}a(x,\cdot )\chi(\cdot )}$
where $\chi=1$ around $0$; thence \eqref{apoly-eq}.
\end{exmp}
\begin{exmp}
Also when $A=a(x,D)$ is of type $1,1$, one can recover its symbol from the
formula
\begin{equation}
  a(x,\xi)= e^{-\im x\cdot \xi} A(e^{\im x\cdot \xi}).
\end{equation}
Here $\cal F e^{\im\dual{\cdot }{\xi}}=(2\pi)^n\delta_{\xi}(\eta)$ has
compact support, so again it follows from \eqref{aFE-eq} that 
(via a suppressed cut-off) 
one has $A(e^{\im\dual{\cdot }{\xi}})
=\dual{\delta_\xi}{e^{\im\dual{x}{\cdot }}a(x,\cdot )}
=e^{\im x\cdot \xi}a(x,\xi)$.
\end{exmp}

\subsection{General Smooth Functions}
To go beyond the smooth functions in \eqref{aFE-eq}, it is shown in this
subsection how one can extend a remark by Bourdaud \cite{Bou87} on singular integral operators,
which shows that every type $1,1$ symbol $a(x,\eta)$ of order $d=0$ induces a map
$\tilde A\colon \cal O_M\to \cal D'$.

Indeed, Bourdaud defined $\tilde Af$ for $f\in\cal O_M(\Rn)$
as the distribution that on $\varphi\in C^\infty_0(\Rn)$ is given by
the following, using the distribution kernel $K$ of $a(x,D)$ and an auxiliary function 
$\chi\in C^\infty_0(\Rn)$ equal to $1$ on a neighbourhood of $\supp\varphi$,
\begin{equation}
  \dual{\tilde Af}{\varphi}= \dual{a(x,D)(\chi f)}{\varphi}
   +\iint K(x,y)(1-\chi(y))f(y)\varphi(x)\,dy\,dx.
  \label{aOC-eq}
\end{equation}
However, to free the discussion from the slow growth in $\cal O_M$, 
one may restate this in terms of the tensor product $1\otimes f$ in 
$\cal S'(\Rn\times \Rn)$ acting on $(\varphi\otimes (1-\chi))K\in \cal
S(\Rn\times \Rn)$, i.e.\ 
\begin{equation}
  \dual{\tilde Af}{\varphi}= \dual{a(x,D)(\chi f)}{\varphi}
   +\dual{1\otimes f}{(\varphi\otimes (1-\chi))K},
  \label{aOC'-eq}
\end{equation}
One advantage here is that both terms obviously make sense as long as
$f$ is smooth and temperate, i.e.\ for every 
$f\in C^\infty (\Rn)\bigcap\cal S'(\Rn)$.

Moreover, for $\varphi$ with support in the interior $\cal C^\circ $ 
of a compact set $\cal C\subset \Rn$ and $\chi=1$ on a neighbourhood of
$\cal C$, the right-hand side of \eqref{aOC'-eq}
gives the same value for any $\tilde \chi\in C^\infty _0$ equal to $1$
around $\cal C$, for in the difference of the right-hand sides equals $0$ since
$\dual{a(x,D)((\chi-\tilde \chi)f)}{\varphi}$ is seen from the kernel relation to equal
$-\dual{1\otimes f}{(\varphi(\tilde \chi-\chi))K}$.

Crude estimates of \eqref{aOC'-eq} now show
that $\tilde Af$ yields a distribution in $\cal D'(\cal C^\circ )$, and
the above $\chi$-independence implies that it coincides in 
$\cal D'(\cal C^\circ \bigcap \cal C_1^\circ )$ with the
distribution defined from another compact set $\cal C_1$. Since
$\Rn=\bigcup \cal C^\circ $, the \emph{recollement de morceaux} theorem yields that
a distribution $\tilde A f\in \cal D'(\Rn)$ is defined by \eqref{aOC'-eq}.

There is also a more explicit formula for $\tilde Af$:
when $\tilde \varphi\in C^\infty _0$ is chosen so that $\tilde \varphi\equiv 1$ around $\cal C$
while $\supp\tilde\varphi$ has a neighbourhood where $\chi=1$,
then $\varphi=\tilde\varphi\varphi$ in \eqref{aOC'-eq} gives, for $x\in\cal C^\circ$,
\begin{equation}
  \tilde A f(x)=a(x,D)(\chi f)(x)+\dual{f}{(\tilde \varphi(x)(1-\chi(\cdot)))K(x,\cdot )}.
\label{aOC''-eq}
\end{equation}
Now $\tilde Af \in C^\infty$ follows, for the first term is in $\cal S$,
and the second coincides in $\cal C^\circ$ with a function in 
$\cal S$, as a corollary to the construction of 
$g\otimes f\in \cal S'(\Rn\times \Rn)$ for $f,g\in \cal S'(\Rn)$.

Post festum, it is seen in \eqref{aOC'-eq} that when $f\to0$ in both $C^\infty$ and
in $\cal S'$, then $\chi f\to0$ in $\cal S$ while $1\otimes f\to 0$ in $\cal S'$.
Therefore $\tilde Af\to0$ in $\cal D'$, which is a basic continuity property of
$\tilde A$. 

By setting $\tilde A$ in relation to Definition~\ref{a11-defn}, the above gives 
the new result that $a(x,D)$ always is a map defined on the
\emph{maximal} set of smooth functions, i.e.\ on $C^\infty \bigcap \cal S'$:

\begin{thm}   \label{aOO-thm}
Every $a(x,D)\in \OP(S^d _{1,1}(\Rn\times \Rn))$ restricts to a map
\begin{equation}
  a(x,D)  \colon C^\infty (\Rn)\bigcap \cal S'(\Rn)\to C^\infty (\Rn),
\label{aOO-eq}
\end{equation}
which locally is given by formula \eqref{aOC''-eq}.
The map \eqref{aOO-eq} is continuous when $C^\infty(\Rn)$ has the usual Fr\'echet space
structure and $\cal S'(\Rn)$ has the strong dual topology. 
\end{thm}

The intersection $C^\infty\bigcap\cal S'$ is topologised by enlarging the set of
seminorms on $C^\infty$ by those on $\cal S'$. Here the latter have the form
$f\mapsto \sup_{\psi\in\cal B}|\dual f\psi|$ for an arbitrary bounded set
$\cal B\subset\cal S$.

\begin{proof}
Let $A_m=\OP(\psi(2^{-m}D_x)a(x,\eta)\psi(2^{-m}\eta))$
so that  $a(x,D)u=\lim_m A_mu$ when $u$ belongs to $D(a(x,D))$.
With $f\in C^\infty \bigcap \cal S'$ and $\varphi, \chi$ as above,
this is the case for $u=\chi f\in C^\infty_0$.

Exploiting the convergence in Remark~\ref{KmS-rem} in \eqref{aOC'-eq}, it is seen that
\begin{equation}
  \dual{\tilde Af}{\varphi}= \lim_m \dual{A_m(\chi f)}{\varphi}
   +\lim_m\iint K_m(x,y)(1-\chi(y))f(y)\varphi(x)\,dy\,dx.
  \label{aOCm-eq}
\end{equation}
Here the integral equals $\dual{A_m(f-\chi f)}{\varphi}$ by the kernel
relation, for $A_m\in \OP(S^{-\infty })$ and $f$ may as an element of $\cal S'$ 
be approached from $C^\infty_0$. So \eqref{aOCm-eq} yields
\begin{equation}
  \dual{\tilde Af}{\varphi}=\lim_m \dual{A_m(\chi f)}{\varphi}
   +\lim_m \dual{A_m(f-\chi f)}{\varphi} =
  \lim_m \dual{A_m f}{\varphi}.
\end{equation}
Thus $A_mf\to \tilde Af$, which by \eqref{aOC'-eq} is independent of $\psi$. 
Hence $\tilde A\subset a(x,D)$ as desired.

That $a(x,D)(C^\infty\cap\cal S')$ is contained in $C^\infty$ now follows from the remarks to \eqref{aOC''-eq}.

When $f\to 0$ in $C^\infty$, then clearly $D^\alpha a(x,D)(f\chi)\to0$ in $\cal S$, 
hence uniformly on $\Rn$. 
It is also straightforward to see that $(\tilde \varphi(x)(1-\chi(\cdot)))K(x,\cdot )$ stays
in a bounded set in $\cal S(\Rn)$ as $x$ runs through $\cal C$. Therefore,  
when $f\to 0$ also in the \emph{strong} dual topology on $\cal S'$, then the second term in
\eqref{aOC''-eq} tends to $0$ uniformly with respect to $x\in \cal C$. 
As $x$-derivatives may fall on $K(x,\cdot)$, the same
argument gives that $\sup_{\cal C}|D^\alpha \tilde Af|\to 0$. 
Hence $f\to 0$ in $C^\infty\bigcap\cal S'$ implies
$D^\alpha \tilde Af\to 0$ in $C^\infty(\Rn)$, which gives the stated continuity property.
\end{proof}

\begin{rem}
  It is hardly a drawback that continuity in Theorem~\ref{aOO-thm} holds for the
  strong dual topology on $\cal S'$, as for
  \emph{sequences} weak and strong convergence are equivalent 
  (a well-known consequence of the fact that $\cal S$ is a Montel space). 
\end{rem}

In view of Theorem~\ref{aOO-thm}, the
difficulties for type $1,1$-operators do not stem from growth at infinity 
for $C^\infty $-functions. Obviously the codomain $C^\infty $
is not contained in $\cal S'$, but this is not just
made possible by the use of $\cal D'$ in Definition~\ref{a11-defn},
it is indeed decisive for the above construction.

In the proof above, the fact that $\tilde A f\in C^\infty$ also follows 
from the pseudo-local property of $a(x,D)$; cf
\cite[Thm.~6.4]{JJ08vfm}. The direct argument above is rather short, though.
In addition to the smoothness, the properties of $a(x,D)f$ can be further sharpened by slow growth
of $f$:

\begin{cor}  \label{aOO-cor}
Every type $1,1$-operator $a(x,D)$ leaves $\cal O_M(\Rn)$ invariant.  
\end{cor}
\begin{proof}
  If $f\in\cal O_M$, then it follows that $(1+|x|)^{-2N}D^\alpha \tilde Af(x) $ is
bounded for sufficiently large $N$, since in the second contribution to \eqref{aOC''-eq} 
clearly $(1+|y|)^{-2N}f(y)$ is in $L_1$ for large $N$:
the resulting factor $(1+|y|)^{2N}$ may be absorbed by $K$, using that 
$r=\op{dist}(\supp\tilde\varphi,\supp(1-\chi))>0$, since for $x\in\supp\tilde\varphi$,
$y\in \supp (1-\chi)$,
\begin{equation}
  \begin{split}
  (1+|y|)^{2N}|D^\alpha_x K(x,y)|&\le 
  (1+|x|)^{2N}\max(1,1/r)^{2N}(r+|x-y|)^{2N}|D^\alpha_x K(x,y)|
\\
  &\le c(1+|x|)^{2N}\sup_{x\in\Rn}\int |(4\lap_\eta)^N \big((\eta+D_x)^\alpha a(x,\eta)\big)|\,d\eta,  
  \end{split}
\label{OM-ineq}
\end{equation}
where the supremum is finite for $2N>d+|\alpha|+n$ (by induction
$(\eta+D_x)^\alpha\colon S^d_{1,1}\to S^{d+|\alpha|}_{1,1}$). Moreover,
$c=\max(1,1/r)^{2N}/(2\pi)^{n}$  can be chosen uniformly for $x\in\Rn$ as it suffices 
to have \eqref{aOC''-eq} with $0\le\chi\le1$ and $\cal C=\Bbar(0,j)$,
$\supp\tilde\varphi=\Bbar(0,j+1)$ and $\chi^{-1}(\{1\})=\Bbar(0,j+2)$ for an arbitrarily large
$j\in\N$, which yields $r=1$, $c\le1$.
Thus $(1+|x|)^{-2N}|D^\alpha \tilde Af(x)|$ is less than
$s_{\alpha,N}\int_{\Rn}(1+|y|)^{-2N}|f(y)|\,dy$
for all $x\in\Rn$, $s_{\alpha,N}$ as the sup in \eqref{OM-ineq}. Hence $\tilde Af\in \cal O_M$.
\end{proof}

\begin{exmp}  \label{nonslow-exmp}
The space $C^\infty (\Rn)\bigcap \cal S'(\Rn)$ clearly contains
functions of non-slow growth, e.g.\ 
\begin{equation}
  f(x)=e^{x_1+\dots +x_n}\cos(e^{x_1+\dots +x_n}).
\end{equation}
Recall that $f\in \cal S'$ because $f=\im D_1 g$ for
$g(x)=\sin(e^{x_1+\dots +x_n})$, which is in $L_\infty \subset \cal S'$.
But $g\notin \cal O_M$, so already for $a(x,D)=\im D_1$ the space $\cal O_M$ cannot
contain the range in Theorem~\ref{aOO-thm}.
\end{exmp}

\begin{rem}
Prior to the $T1$-theorem, David and Journ\'e explained in \cite{DaJo84} how
a few properties of the distribution kernel of a continuous map $T\colon
C^\infty_0(\Rn)\to\cal D'(\Rn)$ makes $T(1)$ well defined modulo
constants; in particular if $T\in \OP(S^0_{1,1})$. 
Bourdaud~\cite{Bou87} used this construction for $\tilde A$,
so by Theorem~\ref{aOO-thm} this extension of $T\in \OP(S^0_{1,1})$ from \cite{DaJo84} 
is contained in Definition~\ref{a11-defn}. 
\end{rem}

\subsection{Conditions along the Twisted Diagonal}   \label{TDC-ssect}
As the first explicit condition formulated for the symbol of a type $1,1$-operator,
H{\"o}rmander \cite{H88} proved that $a(x,D)$ has an extension by continuity 
\begin{equation}
  H^{s+d}\to H^s \quad\text{for every $s\in \R$}   
\end{equation}
whenever $a\in S^d_{1,1}(\Rn\times \Rn)$
fulfils the \emph{twisted diagonal condition}: for some $B\ge 1$ 
\begin{equation}
  \hat a(\xi,\eta)=0 \quad\text{where}\quad
    B(1+|\xi+\eta|)< |\eta|.
  \label{tdc-cnd}
\end{equation}
In detail this means that the partially Fourier transformed symbol 
$\hat a(\xi,\eta):=\cal F_{x\to\xi}a(x,\eta)$
is trivial in a conical neighbourhood of a non-compact part of the twisted
diagonal 
\begin{equation}
  \cal T=\{\,(\xi,\eta)\in \Rn\times \Rn\mid \xi+\eta=0\,\}.  
\end{equation}
Localisations to certain conical neighbourhoods of $\cal T$
were also used by H{\"o}rmander \cite{H88,H89,H97} as
\begin{equation}
  \hat a_{\chi,\varepsilon}(\xi,\eta)
  =\hat a(\xi,\eta)\chi(\xi+\eta,\varepsilon\eta),
  \label{axe-eq}
\end{equation}
whereby the cut-off function $\chi\in C^\infty (\Rn\times \Rn)$ is chosen to satisfy
\begin{gather}
  \chi(t\xi,t\eta)= \chi(\xi,\eta)\quad\text{for}\quad t\ge 1,\ |\eta|\ge 2
\label{chi1-eq}   \\
  \supp \chi\subset \{\,(\xi,\eta)\mid 1\le |\eta|,\ |\xi|\le
|\eta|\,\}
\label{chi2-eq} \\
  \chi=1 \quad\text{in}\quad 
  \{\,(\xi,\eta)\mid 2\le |\eta|,\ 2|\xi|\le |\eta|\,\}.
  \label{chi3-eq}
\end{gather}
Using this, H{\"o}rmander introduced and analysed a
milder condition than the strict vanishing in \eqref{tdc-cnd}.
Namely, for some $\sigma\in\R$
the symbol should satisfy an estimate, for all multiindices $\alpha$ and $0<\varepsilon<1$,
\begin{equation}
  N_{\chi,\varepsilon,\alpha}(a):=
  \sup_{\substack{R>0,\\ x\in \Rn}}R^{-d}\big(
  \int_{R\le |\eta|\le 2R} |R^{|\alpha|}D^\alpha_{\eta}a_{\chi,\varepsilon}
  (x,\eta)|^2\,\frac{d\eta}{R^n}
  \big)^{\frac12}
  \le c_{\alpha,\sigma} \varepsilon^{\sigma+n/2-|\alpha|}.
  \label{Hsigma-eq}
\end{equation}
This is an asymptotic formula for $\varepsilon\to0$. It always holds for $\sigma=0$,
cf.\ \cite[Lem.~9.3.2]{H97}:

\begin{lem}   \label{Heps-lem}
When $a\in S^d_{1,1}(\Rn\times \Rn)$ and $0<\varepsilon\le 1$, then
$a_{\chi,\varepsilon}\in C^\infty $ and 
\begin{gather}
  |D^\alpha_{\eta}D^\beta_x a_{\chi,\varepsilon}(x,\eta)|\le 
   C_{\alpha,\beta}(a)\varepsilon^{-|\alpha|}
  (1+|\eta|)^{d-|\alpha|+|\beta|}
  \\
  \big(
  \int_{R\le |\eta|\le 2R} |D^\alpha_{\eta}a_{\chi,\varepsilon}
  (x,\eta)|^2\,d\eta
  \big)^{1/2}
  \le C_{\alpha} R^d (\varepsilon R)^{n/2-|\alpha|}.
\end{gather} 
The map $a\mapsto a_{\chi,\varepsilon}$ is continuous in $S^d_{1,1}$.
\end{lem}
The last remark on continuity has been inserted here for later reference. 
It is easily verified by observing in the proof of \cite[Lem.~9.3.2]{H97}
(to which the reader is referred) that the constant $C_{\alpha,\beta}(a)$ is
a continuous seminorm in $S^d_{1,1}$.

In case $\sigma>0$ there is a faster convergence to $0$ in \eqref{Hsigma-eq}.
In \cite{H89} this was proved to imply that $a(x,D)$ is bounded as a densely defined map 
\begin{equation}
  H^{s+d}(\Rn)\to H^s(\Rn) \quad\text{for}\quad s>-\sigma.
  \label{Hssigma-eq}
\end{equation}
The reader may consult \cite[Thm.~9.3.5]{H97} for this (whilst
\cite[Thm.~9.3.7]{H97} gives four pages of proof of necessity of 
$s\ge -\sup\sigma$, with supremum over all $\sigma$ satisfying \eqref{Hsigma-eq}).

Consequently, if $\hat a(\xi,\eta)$ is so small along $\cal T$ that \eqref{Hsigma-eq} holds for all $\sigma\in \R$, 
there is boundedness $H^{s+d}\to H^s$ for all $s\in \R$. E.g.\ this is the case when
\eqref{tdc-cnd} holds, for since 
\begin{equation}
  \supp \hat a_{\chi,\varepsilon}\subset 
  \{\,(\xi,\eta)\mid 1+|\xi+\eta|\le 2\varepsilon|\eta| \,\},
\end{equation}
clearly $a_{\chi,\varepsilon}\equiv 0$ for $2\varepsilon<1/B$ then.

\begin{exmp}   \label{Heps-exmp}
For the present paper it is useful to exploit Ching's symbol 
\eqref{ching-eq} to show the existence of symbols
fulfilling \eqref{Hsigma-eq} for a given $\sigma$, at least for $\sigma\in \N$.
To do so one may fix $|\theta|=1$ and take 
some $A(\eta)$ in $C^\infty_0(\{\,\eta\mid \tfrac{3}{4}<|\eta|<\tfrac{5}{4} \,\})$
with a zero of order $\sigma$ at $\theta$, so that Taylor's formula gives
$|A(\eta)|\le c|\eta-\theta|^{\sigma}$ in a neighbourhood of $\theta$:

Indeed, as $\hat a(x,\eta)=(2\pi)^n\sum_{j=0}^\infty 2^{jd}\delta(\xi+2^j\theta)
A(2^{-j}\eta)$, clearly 
\begin{equation}
  a_{\theta,\chi,\varepsilon}(x,\eta)= \sum_{j=0}^\infty 2^{jd} 
  e^{-\im x\cdot 2^j\theta}\chi(\eta-2^j\theta,\varepsilon\eta)
  A(2^{-j}\eta).
\end{equation}
Because $[R,2R]$ is contained in 
$[\tfrac{3}{4}2^{j-1},\tfrac{3}{2}2^{j-1}]\cup [\tfrac{3}{4}2^j,\tfrac{3}{2}2^{j}]$ for some
$j\in\Z$, it suffices to estimate the integral in \eqref{Hsigma-eq} 
only for $R=3\cdot 2^{j-2}$ with $j\ge 1$. Then it involves only the $j$th term, i.e.\ 
\begin{equation}
  \int_{R\le |\eta|\le 2R}|a_{\theta,\chi,\varepsilon}(x,\eta)|^2\,d\eta
=\int_{R\le |\eta|\le 2R} R^{2d} |A(\eta/R)|^2
  |\chi(\eta-R\theta,\varepsilon\eta)|^2\,d\eta. 
  \label{R2j-eq}
\end{equation}
By the choice of $\chi$, the integrand is $0$ unless 
$|\eta-R\theta|\le \varepsilon |\eta|\le 2\varepsilon R$ and $1\le \varepsilon R$, 
so for small $\varepsilon$,
\begin{equation}
  \int_{R\le |\eta|\le 2R}|a_{\theta,\chi,\varepsilon}(x,\eta)|^2\,d\eta
\le \nrm{\chi}{\infty }^2 R^{n+2d}
\int_{|\zeta-\theta|\le 2\varepsilon}
  c|\zeta-\theta|^{2\sigma}\,d\zeta
\le c'\varepsilon^{2\sigma+n} R^{n+2d}.
\end{equation}
Applying $(RD_{\eta})^\alpha$ before integration, 
$(RD_{\eta})^\gamma$ may fall on $A(\eta/R)$, which
lowers the degree and yields (at most) $\varepsilon^{n/2+\sigma-|\gamma|}$. 
In the factor 
$(RD_{\eta})^{\alpha-\gamma}\chi(\eta-R\theta,\varepsilon\eta)$
the homogeneity of degree $-|\alpha-\gamma|$ applies for 
$\varepsilon R\ge2$ and yields a bound in terms of finite suprema over
$B(\theta,2)\times B(0,2)$, hence is $\cal O(1)$;
else $\varepsilon R<2$ so the factor is 
$\cal O(R^{|\alpha-\gamma|})=\cal O(\varepsilon^{|\gamma|-|\alpha|})$
when non-zero, as both entries are in norm less than $4$ then.
Altogether this verifies \eqref{Hsigma-eq}.

A lower bound of \eqref{R2j-eq} by $c\varepsilon^{2\sigma+n}
R^{n+2d}$ is similar (cf.\ \cite[Ex.~9.3.3]{H97} for $\sigma=0=d$) when 
$|A(\eta)|\ge c_0|\eta-\theta|^{\sigma}$, which is obtained by taking
$A$ as a localisation of $|\eta-\theta|^{\sigma}$ for even $\sigma$ 
(so $A\in C^\infty $). This implies that
\eqref{Hsigma-eq} does not hold for larger values of $\sigma$ for this $a_{\theta}(x,\eta)$.
\end{exmp}

\section{Pointwise Estimates}   \label{pe-sect}
A crucial technique in this paper will be to estimate 
$|a(x,D)u(x)|$ at an arbitrary point $x\in\Rn$.
Some of recent results on this by the author \cite{JJ11pe} are recalled here and
further elaborated in Section~\ref{cutoff-ssect} with an estimate of frequency
modulated operators.

\subsection{The Factorisation Inequality}   \label{revw-ssect}
First of all, by \cite[Thm.~4.1]{JJ11pe}, 
when $\supp\hat u\subset \Bbar(0,R)$, the action on $u$ by
the operator $a(x,D)$ can be \emph{separated} from $u$ at the cost of an estimate,
which is the \emph{factorisation inequality}
\begin{equation}
  |a(x,D)u(x)|\le F_a(N,R;x) u^*(N,R;x).
  \label{Fau*-eq}
\end{equation} 
Hereby $u^*(x)=u^*(N,R;x)$ denotes the maximal function of Peetre--Fefferman--Stein type,
\begin{equation}
  u^*(N,R;x)=\sup_{y\in \Rn}\frac{|u(x-y)|}{(1+R|y|)^N}
  =\sup_{y\in \Rn}\frac{|u(y)|}{(1+R|x-y|)^N} .
  \label{u*-eq}
\end{equation}
The parameter $N$ is often chosen to satisfy $N\ge \order \hat u$.

The $a$-factor $F_a$, also called the symbol factor, only depends on $u$ in
a vague way, i.e.\ only through the $N$
and $R$ in \eqref{u*-eq}. It is related to the distribution kernel of $a(x,D)$ as
\begin{equation}
  F_a(N,R;x)= \int_{\Rn} (1+R|y|)^N 
  |\cal F^{-1}_{\eta\to y}(a(x,\eta)\chi(\eta ))|\,dy,
  \label{Fa-id}
\end{equation}
where $\chi\in C^\infty_0(\Rn)$ should equal $1$ in a neighbourhood of 
$\supp\hat u$, or of $\bigcup_{x}\supp a(x,\cdot )\hat u(\cdot)$.

In \eqref{Fau*-eq} both factors are easily controlled. 
For one thing the non-linear map $u\mapsto u^*$ has long been known to have bounds with respect to
the $L_p$-norm; cf.\ \cite[Thm.~2.6]{JJ11pe} for an elementary proof.
But in the present paper it is more important that $u^*(x)$ is polynomially bounded thus:
$|u(y)|\le c(1+|y|)^N\le c(1+R|y-x|)^N(1+|x|)^N$ holds according
to the Paley--Wiener--Schwartz Theorem if $N\ge \order \hat u$ and $R\ge 1$,
which by \eqref{u*-eq} implies
\begin{equation}
  u^*(N,R;x)\le c (1+|x|)^N,\qquad x\in \Rn.
  \label{u*PWS-eq}
\end{equation}
Here it is first recalled that every $u\in \cal S'$ has finite order since, for $\psi\in \cal S$,
\begin{align}
  |\dual{u}{\psi}|&\le c p_N(\psi),
  \label{upN-ineq}
\\
  p_N(\psi)&=\sup\{\,(1+|x|)^N|D^\alpha u(x)|\mid x\in \Rn,\ |\alpha|\le N
\,\}.
  \label{pN-eq}
\end{align}
Indeed, since $(1+|x|)^N$ is finite on $\supp\psi$ for $\psi\in C^\infty_0$,
$u$ is of order $N$. To avoid a discussion of the converse, it will throughout 
be convenient to call the least integer $N$ fulfilling \eqref{upN-ineq} 
the \emph{temperate} order of $u$, written $N=\order_{\cal S'}(u)$.

Returning to \eqref{u*PWS-eq}, when the compact spectrum of $u$ results from Fourier
multiplication, then the below $\cal O(2^{kN})$-information 
on the constant will be used repeatedly in the present paper.

\begin{lem}   \label{u*-lem} 
Let $u\in \cal S'(\Rn)$ be arbitrary and $N\ge \order_{\cal S'}( \hat u)$.
When $\psi\in C^\infty_0(\Rn)$ has support in $\Bbar(0,R)$, then 
$w=\psi(2^{-k}D)u$ fulfils
\begin{equation} \label{w*PWS-eq}
  w^*(N,R2^k;x)\le C 2^{kN}(1+|x|)^N,
\qquad k\in \N_0,
\end{equation}
for a constant $C$ independent of $k$. 
\end{lem}
\begin{proof}
As $\psi(2^{-k}D)u(x)=\dual{\hat u}{\psi(2^{-k}\cdot ) e^{\im \dual{x}{\cdot }}(2\pi)^{-n}}$, 
continuity of $\hat u\colon \cal S\to\C$ yields
\begin{equation}
  |w(x)|\le c \sup\bigl\{\,(1+|\xi|)^N
    |D^\alpha_{\xi}(\psi(2^{-k}\xi)e^{\im \dual{x}{\xi}})| \bigm| 
    \xi\in \Rn,\ |\alpha|\le N\,\bigr\}.
\end{equation}
Since $(1+|\xi|)^N|D^\alpha\psi(2^{-k}\xi)|\le c'2^{k(N-|\alpha|)}$,
Leibniz' rule gives that $|w(x)|\le c''2^{kN}(1+|x|)^N$. 
Proceeding as before the lemma, the claim follows with 
$C=c''\max(1,R^{-N})$.
\end{proof}

Secondly, for the $a$-factor  in \eqref{Fa-id} one has $F_a\in C(\Rn)\cap L_\infty (\Rn)$ and an estimate
highly reminiscent of the Mihlin--H{\"o}rmander conditions for Fourier multipliers:

\begin{thm}
  \label{Fa-thm}
Assume $a(x,\eta)$ is in $S^d_{1,1}(\Rn\times\Rn)$ and let $F_a(N,R;x)$ be given by  
\eqref{Fa-id} for parameters $R,N>0$, with the auxiliary function
taken as $\chi=\psi(R^{-1}\cdot)$ for $\psi\in C^\infty_0(\Rn)$ equalling
$1$ in a set with non-empty interior. Then one has for all $x\in \Rn$ that
\begin{equation}
 0\le F_a(x) \le c_{n,N} \sum_{|\alpha|\le [N+\frac{n}{2}]+1} 
  \Big(\int_{R\supp \psi} |R^{|\alpha|}D^{\alpha}_\eta a(x,\eta)|^2
    \,\frac{d\eta}{R^n}
  \Big)^{1/2}.
  \label{FaMH-eq}
\end{equation}
\end{thm}

For the elementary proof the reader may consult \cite{JJ11pe}; cf.\ Theorem~4.1 and Section~6 there.
A further analysis of how $F_a$ depends on $a(x,\eta)$ and $R$ is a special case of \cite[Cor.~4.3]{JJ11pe}: 

\begin{cor}
  \label{Fa-cor} 
Assume $a\in S^{d}_{1,1}(\Rn\times\Rn)$ and let $N$, $R$ and $\psi$ be
as in Theorem~\ref{Fa-thm}.
When $R\ge 1$ there is a seminorm $p$ on $S^d_{1,1}$ and a constant
$c>0$, that depends only on $n$, $N$ and $\psi$, such that
\begin{equation}
 0\le F_a(x) \le c_1 p(a) R^{\max(d,[N+n/2]+1)} \quad\text{for all}\quad x\in
\Rn. 
  \label{Rmax-eq}
\end{equation}
If $\psi(\eta)\ne0$ only holds in a corona $0<\theta_0 \le|\eta|\le \Theta_0$,
and $\psi(\eta)=1$ holds for $\theta_1\le |\eta|\le \Theta_1$, then
\begin{equation}
  0\le F_a(x)\le c_0 p(a)R^{d} \quad\text{for all}\quad x\in \Rn,
  \label{Rd-eq}
\end{equation}
whereby $c_0=c_1\max(1,\theta_0^{d-N-[n/2]-1},\theta_0^d)$.
\end{cor}

The above asymptotics is $\cal O(R^d)$ for $R\to\infty $  if $d$ is large.
This can be improved when $a(x,\eta)$ is modified by removing the low
frequencies in the $x$-variable (cf.\ the $a^{(3)}$-term in Section~\ref{corona-sect} below).
In fact, with a second spectral quantity $Q>0$, 
the following is contained in \cite[Cor.~4.4]{JJ11pe}:

\begin{cor}
  \label{Fa0-cor}
When $a_{Q}(x,\eta)=\varphi(Q^{-1}D_x)a(x,\eta)$ for some $a\in
S^d_{1,1}(\Rn\times \Rn)$ and $\varphi\in C^\infty_0(\Rn)$ with $\varphi=0$
in a neighbourhood of $\xi=0$, then there is a seminorm $p$ on $S^d_{1,1}$
and constants $c_M$, depending only on $M$, $n$, $N$, $\psi$ and $\varphi$,
such that 
\begin{equation}
  0\le F_{a_{Q}}(N,R;x)\le c_M p(a)Q^{-M} R^{\max(d+M,[N+n/2]+1)}
\quad\text{for}\quad M,Q,R>0.
\end{equation}
Here $d+M$ can replace $\max(d+M,[N+n/2]+1)$ when the auxiliary function $\psi$ in
$F_{a_Q}$ fulfils the corona condition in Corollary~\ref{Fa-cor}.
\end{cor}

\begin{rem}
  \label{Fa-rem}
By the proofs in \cite{JJ11pe}, the seminorms in
Corollaries~\ref{Fa-cor} and \ref{Fa0-cor} may be chosen
in the same way for all $d$, namely $p(a)=\sum_{|\alpha|\le [N+n/2]+1}
p_{\alpha,0}(a)$; cf.\ \eqref{pab-eq}. 
\end{rem}

\subsection{Estimates of Frequency Modulated Operators}
  \label{cutoff-ssect}
The results in the previous section easily give the following, 
which later in Sections~\ref{corona-sect} and \ref{split-sect} will be 
used repeatedly.

\begin{prop}   
  \label{cutoff-prop}
For $a(x,\eta)$ in $S^d_{1,1}(\Rn\times \Rn)$, $v\in \cal S'(\Rn)$ and arbitrary $\Phi$, $\Psi\in
C^\infty_0(\Rn)$, for which $\Psi$ is constant in a neighbourhood of the
origin and with its support in $\Bbar(0,R)$ for some $R\ge 1$, 
there is for $k\in \N_0$ and $N\ge \order_{\cal S'}( \cal Fv)$, cf.\ \eqref{pN-eq}, a polynomial bound
\begin{equation}
  \big|\OP\big(\Phi(2^{-k}D_x)a(x,\eta)\Psi(2^{-k}\eta)\big)v(x)\big|
  \le  C(k)(1+|x|)^N,
  \end{equation}
whereby
  \begin{equation}
 C(k)=
  \begin{cases}
    c 2^{k(N+d)_+}&\ \text{ for } N+d\ne0,
\\ 
     c k &\ \text{ for } N+d=0.
  \end{cases}
\end{equation}
For $0\notin\supp\Psi$ this may be sharpened to $C(k)=c 2^{k(N+d)}$
for all values of $N+d$.
\end{prop}

\begin{proof}
In this proof it  is convenient to let
$a^k(x,\eta)=\Phi(2^{-k}D_x)a(x,\eta)$ and $v^k=\Psi(2^{-k}D)v$.
Then the factorisation inequality \eqref{Fau*-eq} gives
\begin{equation}
  |a^{k}(x,D)v^k(x)|\le 
  F_{a^{k}}(N,R2^k;x) \cdot  (v^k)^*(N,R2^k;x).
\end{equation}
Since $N\ge \order_{\cal S'}(\hat v)$, Lemma~\ref{u*-lem} gives
$(v^k)^*(N,R2^k;x)  \le C 2^{kN}(1+|x|)^N$, $x\in \Rn$.

In case $0\notin \supp\Psi$,  the auxiliary function $\chi=\psi(\cdot /(R2^k))$
used in $F_{a^{k}}$, cf.\ Theorem~\ref{Fa-thm}, can be so chosen that it fulfils the 
corona condition in Corollary~\ref{Fa-cor}; e.g.\ it is possible to have
$\Theta_1=1$ and $\theta_1=r/R$ when $\Psi\equiv0$ on $B(0,r)$.
Since Remark~\ref{Fa-rem} implies
$p(a^k)\le p(a)\int |\cal F^{-1}\Phi(y)|\,dy$,
\begin{equation}
 0\le F_{a^{k}}(N,R2^k;x)\le c_0\nrm{\cal F^{-1}\Phi}{1}p(a)R^d2^{kd}.
\end{equation} 
When combined with the above, this inequality yields the claim in case $0\notin \Psi$.  

In the general case one has $v^k=v_k+v_{k-1}+\dots +v_1+v^0$, whereby $v_j$ denotes the
difference $v^j-v^{j-1}=\Psi(2^{-j}D)v-\Psi(2^{-j+1}D)v$. Via \eqref{Fau*-eq} this gives the
starting point
\begin{equation}
  |a^k(x,D)v^k(x)|
  \le |a^k(x,D)v^0(x)|+\sum_{j=1}^{k} F_{a^k}(N,R2^j;x)v_j^*(N,R2^j;x).
  \label{tele-ineq}
\end{equation}
As $\tilde \Psi=\Psi-\Psi(2\cdot )$  
does not have $0$ in its support, the above shows
that for $j=1,\dots ,k$ one has
$F_{a^{k}}(N,R2^j;x)\le c_0\nrm{\cal F^{-1}\Phi}{1}p(a)R^d2^{jd}$.
Lemma~\ref{u*-lem} yields polynomial bounds of $v_j^*$, say with a
constant $C'$,
so the sum on the right-hand side of \eqref{tele-ineq} is estimated,
for $d+N\ne0$, by
\begin{equation}
  \sum_{j=1}^k c_0C' R^d p(a)2^{j(N+d)}(1+|x|)^N
 \le 
   \frac{c_0C' R^d}{2^{|d+N|}-1}p(a)(1+|x|)^N 2^{(k+1)(N+d)_+}.
\end{equation}
In case $d+N=0$ the $k$ bounds are equal.

The remainder in \eqref{tele-ineq} fulfils
$|a^k(x,D)v^0(x)|\le c_1R^{N'}(1+|x|)^N$ for a large $N'$;  
cf.\ the first part of Corollary~\ref{Fa-cor} and Lemma~\ref{u*-lem}.
Altogether 
$|a^k(x,D)v^{k}(x)|\le C(k)(1+|x|)^{N}$.
\end{proof}

\section{Adjoints of Type $1,1$-Operators}   \label{adj-sect}
For classical pseudo-differential operators $a(x,D)\colon {\cal S}\to{\cal S'}$
it is well known that the adjoint 
$a(x,D)^*\colon{\cal S}\to{\cal S'}$ has symbol $a^*(x,\eta)=e^{\im D_x\cdot  D_\eta}\overline{a(x,\eta)}$,
and that $a\mapsto a^*$ sends e.g.\ $S^d_{1,0}$ into itself.

\subsection{The Basic Lemma}   \label{tdc-ssect}
In order to show that the twisted diagonal condition \eqref{tdc-cnd} 
also implies continuity $a(x,D)\colon \cal S'\to\cal S'$, 
a basic result on the adjoint symbols is recalled from
\cite{H88} and \cite[Lem.~9.4.1]{H97}:

\begin{lem}
  \label{GB-lem}
When $a(x,\eta)$ is in $S^{d}_{1,1}(\Rn\times \Rn)$ and for some 
$B\ge 1$ satisfies the twisted diagonal condition \eqref{tdc-cnd},
then the adjoint symbol 
$b(x,\eta)=e^{\im D_x\cdot  D_\eta}\overline{a(x,\eta)}$ is in 
$S^{d}_{1,1}(\Rn\times \Rn)$ in this case and
\begin{equation}
  \hat b(\xi,\eta)=0 \quad\text{when}\quad |\xi+\eta|>B(|\eta|+1).
\end{equation}
Moreover,
\begin{equation}
  |D^\alpha_\eta D^\beta_x b(x,\eta)|\le 
  C_{\alpha\beta}(a)B(1+B^{d-|\alpha|+|\beta|})(1+|\eta|)^{d-|\alpha|+|\beta|},
  \label{tdc-ineq}
\end{equation}
for certain seminorms $C_{\alpha\beta}$ that are continuous on 
$S^{d}_{1,1}(\Rn\times \Rn)$ and do not depend on $B$.
\end{lem}

The twisted diagonal condition \eqref{tdc-cnd} implies that 
$a^*(x,D)=b(x,D)$ is a map $\cal S\to\cal S$, as it is of type $1,1$ by Lemma~\ref{GB-lem}, so then 
$a(x,D)$ has the continuous linear \emph{extension} $b(x,D)^*\colon \cal S'\to\cal S'$.
It is natural to expect that this coincides with the definition 
of $a(x,D)$ by vanishing frequency modulation:

\begin{prop}
  \label{GB-prop}
If $a(x,\eta)\in S^d_{1,1}(\Rn\times \Rn)$ fulfils 
\eqref{tdc-cnd}, then $a(x,D)$ is a continuous linear map $\cal S'(\Rn)\to\cal
S'(\Rn)$ that equals the adjoint of $b(x,D)\colon \cal S(\Rn)\to\cal
S(\Rn)$, when $b(x,\eta)$ is the adjoint symbol as in Lemma~\ref{GB-lem}.
\end{prop}
\begin{proof}
When $\psi\in C^\infty_0(\Rn)$ is such that $\psi=1$ in a neighbourhood of
the origin, a simple convolution estimate (cf.\ \cite[Lem.~2.1]{JJ08vfm})
gives that in the topology of $S^{d+1}_{1,1}$,
\begin{equation}
  \psi(2^{-m}D_x)a(x,\eta)\psi(2^{-m}\eta)\to a(x,\eta)
  \quad\text{for}\quad m\to\infty .
\end{equation}
Since the supports of the partially Fourier transformed symbols
\begin{equation}
  \psi(2^{-m}\xi)\cal F_{x\to\xi}a(\xi,\eta)\psi(2^{-m}\eta),\quad m\in \N,
\end{equation}
are contained in $\supp\cal F_{x\to\xi}a(\xi,\eta)$, clearly this sequence
also fulfils \eqref{tdc-cnd} for the same $B$. 
As the passage to adjoint symbols by \eqref{tdc-ineq}
is continuous from the metric subspace of $S^{d}_{1,1}$ 
fulfilling \eqref{tdc-cnd} to $S^{d+1}_{1,1}$, one therefore has that
\begin{equation}
  b_m(x,\eta):= e^{\im D_x\cdot D_\eta}
  (\overline{\psi(2^{-m}D_x)a(x,\eta)\psi(2^{-m}\eta)})
  \xrightarrow[m\to\infty ]{~}
     e^{\im D_x\cdot D_\eta}\overline{a(x,\eta)}=: b(x,\eta).
  \label{*conv-eq}
\end{equation}
Combining this with the fact that $b(x,D)$ as an operator on the Schwartz
space depends continuously on the symbol, 
one has for $u\in \cal S'(\Rn)$, $\varphi\in \cal S(\Rn)$,
\begin{equation}
  \begin{split}
    \scal{b(x,D)^*u}{\varphi}              
&= \scal{u}{\lim_{m\to\infty }\OP(b_m(x,\eta))\varphi}
\\
  &= \lim_{m\to\infty }
     \scal{\OP(\psi(2^{-m}D_x)a(x,\eta)\psi(2^{-m}\eta))u}{\varphi}.
  \end{split}
\end{equation}
As the left-hand side is independent of $\psi$ the limit 
in \eqref{aPsi-eq} is so, hence
the definition of $a(x,D)$ gives that every 
$u\in \cal S'(\Rn)$ is in $D(a(x,D))$ and $a(x,D)u=b(x,D)^*u$ as claimed.
\end{proof}

The mere extendability to 
$\cal S'$ under the twisted diagonal condition \eqref{tdc-cnd} could
have been observed already in \cite{H88,H97}, 
but the above result seems to be
the first sufficient condition for a type $1,1$-operator to
be \emph{defined} on the entire $\cal S'(\Rn)$.

\subsection{The Self-Adjoint Subclass $\mathbf{\widetilde S}^{\mathbf{d}}_{\mathbf{1}\mathbf{1}}$}
   \label{small-ssect}
Proposition~\ref{GB-prop} shows that \eqref{tdc-cnd} suffices for $D(a(x,D))=\cal S'$.
But \eqref{tdc-cnd} is too strong to be necessary; a vanishing to infinite
order along $\cal T$ should suffice.

In this section, the purpose is to prove that 
$a(x,D)\colon \cal S'\to\cal S'$ is continuous if more
generally the twisted diagonal condition of order $\sigma$, 
that is \eqref{Hsigma-eq}, holds for all $\sigma\in \R$.

This will supplement H{\"o}rmander's investigation in \cite{H88,H89,H97},
from where the main ingredients are recalled.
Using \eqref{axe-eq} and $\cal F_{x\to\xi}$ one has in $\cal S'(\Rn\times \Rn)$, 
\begin{equation}
  a(x,\eta)= (a(x,\eta)-a_{\chi,1}(x,\eta))+
  \sum_{\nu=0}^\infty (a_{\chi,2^{-\nu}}(x,\eta)-a_{\chi,2^{-\nu-1}}(x,\eta)).
  \label{Hsum-eq}
\end{equation}
Here the first term $a(x,\eta)-a_{\chi,1}(x,\eta)$
fulfils \eqref{tdc-cnd} for $B=1$, 
so Proposition~\ref{GB-prop} applies to it.
Introducing $e_\varepsilon(x,D)$ like in \cite[Sect.~9.3]{H97} as
\begin{equation}
  \hat e_\varepsilon(x,\eta)=
   \hat a_{\chi,\varepsilon}(\xi,\eta)-\hat a_{\chi,\varepsilon/2}(\xi,\eta)
   =(\chi(\xi+\eta,\varepsilon\eta)-\chi(\xi+\eta,\varepsilon\eta/2))
    \hat a(x,\eta),
\end{equation}
it is useful to infer from the choice of $\chi$ that
\begin{equation}
  \supp\hat e_\varepsilon \subset 
  \Set{(\xi,\eta)}{
    \tfrac{\varepsilon}{4}|\eta|\le\max(1,|\xi+\eta|)\le \varepsilon|\eta|}.
  \label{Feep-eq}
\end{equation}
In particular this yields that $\hat e_\varepsilon=0$ when
$1+|\xi+\eta|<|\eta|\varepsilon/4$, so $e_\varepsilon$ fulfils
\eqref{tdc-cnd} for $B=4/\varepsilon$. Hence the terms $e_{2^{-\nu}}$ in
\eqref{Hsum-eq} do so for $B=2^{\nu+2}$.

The next result characterises the $a\in S^d_{1,1}$ for which the adjoint
symbol $a^*$ is again in $S^d_{1,1}$; cf.\ the below condition \eqref{a*-cnd}. 
As adjoining is an involution, these symbols constitute the class
\begin{equation} \label{Stilde-id}
  \widetilde S^d_{1,1}:= S^d_{1,1}\cap (S^d_{1,1})^* .
\end{equation}

\begin{thm}   \label{a*-thm}
When $a(x,\eta)$ is a symbol in $S^d_{1,1}(\Rn\times \Rn)$ the following
properties are equivalent:
\begin{rmlist}
    \item   \label{a*-cnd}
    The adjoint symbol $a^*(x,\eta)$ is also in $S^d_{1,1}(\Rn\times \Rn)$.
  \item   \label{orderN-cnd}
  For arbitrary $N>0$ and $\alpha$, $\beta$ there is some constant
$C_{\alpha,\beta,N}$ such that 
\begin{equation}
  |D^\alpha_\eta D^\beta_x a_{\chi,\varepsilon}(x,\eta)|\le C_{\alpha,\beta,N}
  \varepsilon^{N}(1+|\eta|)^{d-|\alpha|+|\beta|}
  \quad\text{for}\quad 0<\varepsilon<1.
\end{equation} 
  \item   \label{sigma-cnd}
  For all $\sigma\in\R$ there is a constant $c_{\alpha,\sigma}$ such that
for $0<\varepsilon<1$
\begin{equation}
  \sup_{R>0,\ x\in\Rn} R^{|\alpha|-d}
  \big( \int_{R\le |\eta|\le 2R}
   |D^\alpha_{\eta} a_{\chi,\varepsilon}(x,\eta)|^2\,\frac{d\eta}{R^n}
  \big)^{1/2}
  \le c_{\alpha,\sigma} \varepsilon^{\sigma+\tfrac{n}{2}-|\alpha|}.
\end{equation}
\end{rmlist}
In the affirmative case $a\in \widetilde S^{d}_{1,1}$, cf.\ \eqref{Stilde-id}, and $a^* $ fulfils an estimate
\begin{equation}
  |D^\alpha_\eta D^\beta_x a^*(x,\eta)| \le 
  (C_{\alpha,\beta}(a)+C'_{\alpha,\beta,N})(1+|\eta|)^{d-|\alpha|+|\beta|}
\end{equation}
for a certain continuous seminorm $C_{\alpha,\beta}$ on $S^d_{1,1}(\Rn\times \Rn)$ and some finite sum
$C'_{\alpha,\beta,N}$ of constants fulfilling the inequalities in \eqref{orderN-cnd}.
\end{thm}

It should be observed from \eqref{a*-cnd} that $a(x,\eta)$ fulfils condition \eqref{orderN-cnd} or
\eqref{sigma-cnd} if and only if 
$a^*(x,\eta)$ does so\,---\,whereas neither \eqref{orderN-cnd} nor \eqref{sigma-cnd} make this obvious.
But \eqref{orderN-cnd} immediately gives the (expected) inclusion 
$\widetilde S^d_{1,1}\subset \widetilde S^{d'}_{1,1}$  for $d'>d$.
Condition \eqref{sigma-cnd} is
close in spirit to the Mihlin--H{\"o}rmander multiplier theorem,
and it is useful for the estimates shown later in Section~\ref{split-sect}.

\begin{rem}  \label{Heps-rem}
Conditions \eqref{orderN-cnd}, \eqref{sigma-cnd} 
both hold either for all $\chi$ satisfying \eqref{Hsigma-eq} or for none, 
for \eqref{a*-cnd} does not depend on $\chi$.
It suffices to verify \eqref{orderN-cnd} or \eqref{sigma-cnd} for
$0<\varepsilon<\varepsilon_0$ for some convenient 
$\varepsilon_0\in \,]0,1[\,$. This is implied by Lemma~\ref{Heps-lem} 
since every power $\varepsilon^p$ is bounded on the interval
$[\varepsilon_0,1]$. 
\end{rem}

Theorem~\ref{a*-thm} was undoubtedly known to H{\"o}rmander, for he stated the equivalence of \eqref{a*-cnd} 
and \eqref{orderN-cnd} explicitly in \cite[Thm.~4.2]{H88} and
\cite[Thm.~9.4.2]{H97} and gave brief remarks on \eqref{sigma-cnd} in the latter. 
Equivalence with continuous extensions $H^{s+d}\to H^s$ for all $s\in \R$ was also shown.
However, the exposition there left a considerable burden of verification to the reader.

Moreover, Theorem~\ref{a*-thm} was used without proof in a main $L_p$-theorem in \cite{JJ11lp},
and below in Section~\ref{selfad-sssect} a corollary to the proof will follow and decisively enter the first
proof of $\cal S'$-continuity.
Hence full details on the main result in Theorem~\ref{a*-thm} should be in order here:

\subsubsection{Equivalence of \eqref{orderN-cnd} and \eqref{sigma-cnd}}
That \eqref{orderN-cnd} implies \eqref{sigma-cnd} is seen at once by
insertion, taking $\beta=0$ and $N=\sigma+\tfrac{n}{2}-|\alpha|$.

Conversely, note first that $|\xi+\eta|\le \varepsilon |\eta|$ in the
spectrum of  $a_{\chi,\varepsilon}(\cdot ,\eta)$. 
That is, $|\xi|\le (1+\varepsilon)|\eta|$
so Bernstein's inequality gives 
\begin{equation}
  |D^\beta_x D^\alpha_\eta a_{\chi,\varepsilon}(x,\eta)|\le 
  ((1+\varepsilon)|\eta|)^{|\beta|}\sup_{x\in \R}
  |D^\alpha_\eta a_{\chi,\varepsilon}(x,\eta)|.
\end{equation} 
Hence $C_{\alpha,\beta,N}=2^{|\beta|}C_{\alpha,0,N}$ is possible,
so it suffices to prove
\eqref{sigma-cnd}$\implies$\eqref{orderN-cnd} only for $\beta=0$.

For the corona $1\le |\zeta|\le 2$ Sobolev's lemma gives for $f\in C^\infty
(\Rn)$,
\begin{equation}
  |f(\zeta)|\le c_1 (\sum_{|\beta|\le [n/2]+1} \int_{1\le |\zeta|\le 2}
   |D^\beta f(\zeta)|^2\,d\zeta)^{1/2}.
\end{equation}
Substituting  $D^{\alpha}_\eta a_{\chi,\varepsilon}(x,R\zeta)$
and $\zeta=\eta/R$, whereby $R\le |\eta|\le 2R$, $R>0$, yields
\begin{equation}
  \begin{split}
    |D^{\alpha}_\eta a_{\chi,\varepsilon}(x,\eta)|
&\le c_1 (\sum_{|\beta|\le [n/2]+1} R^{2|\beta|}\int_{R\le |\eta|\le 2R}
   |D^{\alpha+\beta}_{\eta} a_{\chi,\varepsilon}(x,\eta)|^2\,
  \frac{d\eta}{R^n})^{1/2}
\\
 &\le c_1 (\sum_{|\beta|\le [n/2]+1}R^{2d-2|\alpha|}C_{\alpha+\beta,\sigma}^2
   \varepsilon^{2(\sigma+\tfrac{n}{2}-|\alpha|-|\beta|)})^{1/2}
\\
 &\le c_1 (\sum_{|\beta|\le [n/2]+1}C_{\alpha+\beta,\sigma}^2
   )^{1/2} \varepsilon^{\sigma-1-|\alpha|} R^{d-|\alpha|}.
  \end{split}
\end{equation}
Here $R^{d-|\alpha|}\le (1+|\eta|)^{d-|\alpha|}$ for
$d\ge |\alpha|$, that leads to \eqref{orderN-cnd} 
as $\sigma\in \R$ can be arbitrary.

For $|\alpha|>d$ it is first noted that, by the support condition on $\chi$, 
one has $a_{\chi,\varepsilon}(x,\eta)\ne 0$ only for $2R\ge |\eta|\ge \varepsilon^{-1}>1$.
But $R\ge 1/2$ yields
$R^{d-|\alpha|}\le (\tfrac{1}{3}(\tfrac{1}{2}+2R))^{d-|\alpha|}\le 
  6^{|\alpha|-d}(1+|\eta|)^{d-|\alpha|}$,
so \eqref{orderN-cnd} follows from the above.

\subsubsection{The implication
\eqref{orderN-cnd}$\implies$\eqref{a*-cnd} and the estimate}
The condition \eqref{orderN-cnd} is exploited for each term in the
decomposition \eqref{Hsum-eq}. Setting
$b_\nu(x,\eta)=e^*_{2^{-\nu}}(x,\eta)$ it follows from Lemma~\ref{GB-lem}
that $b_\nu$ is in $S^d_{1,1}$ by the remarks after \eqref{Feep-eq}, cf
\eqref{Hsum-eq} ff, and \eqref{tdc-ineq} gives
\begin{equation}
  |D^\alpha_\eta D^\beta_x b_\nu(x,\eta)|\le 
  C_{\alpha,\beta}(e_\nu)2^{\nu+2} (1+2^{(\nu+2)(d-|\alpha|+|\beta|)})
  (1+|\eta|)^{d-|\alpha|+|\beta|}.
\end{equation}
Now \eqref{orderN-cnd} implies that 
$C_{\alpha,\beta}(a_{\chi,2^{-\nu}})\le 
C'_{\alpha,\beta,N}2^{-\nu N}$ for all $N>0$
(with other contants $C'_{\alpha,\beta,N}$ as the seminorms
$C_{\alpha,\beta}$ may contain derivatives of higher order than
$|\alpha|$ and $|\beta|$). Hence
$C_{\alpha,\beta}(e_{2^{-\nu}})\le 
C'_{\alpha,\beta,N}2^{1-\nu N}$.
It follows from this that $\sum b_\nu$ converges to some $b$
in $S^d_{1,1}$ (in the Fr\'echet topology of this space), so that
$a^*(x,\eta)=b(x,\eta)$ is in $S^d_{1,1}$.
More precisely, \eqref{tdc-ineq} and the above yields for
$N=2+(d-|\alpha|+|\beta|)_+$ 
\begin{equation}
  \begin{split}
  \frac{|D^\alpha_\eta D^\beta_x a^*(x,\eta)|}{(1+|\eta|)^{d-|\alpha|+|\beta|}}
  &\le 2^N C_{\alpha,\beta}(a-a_{\chi,1}) 
   +\sum_{\nu=0}^\infty C_{\alpha,\beta}(e_{2^{-\nu}})2^{\nu+2}
     (1+2^{(\nu+2)(d-|\alpha|+|\beta|)_+})
\\
  &\le 2^N C_{\alpha,\beta}(a-a_{\chi,1}) 
   +\sum_{\nu=0}^\infty 16 C'_{\alpha,\beta,N}2^{-\nu(N-1)}
     2^{(\nu+2)(d-|\alpha|+|\beta|)_+}
\\
  &\le 2^N C_{\alpha,\beta}(a-a_{\chi,1}) 
   + 4^{N+2} C'_{\alpha,\beta,N}.
  \end{split}
  \label{a*-est}
\end{equation}
Invoking continuity from Lemma~\ref{Heps-lem} in the first term, the
last part of the theorem follows. 

\subsubsection{Verification of \eqref{a*-cnd}$\implies$\eqref{orderN-cnd}}
It suffices to derive another decomposition 
\begin{equation}
  a=A+\sum_{\nu=0}^\infty a_\nu,
  \label{aAnu-eq}
\end{equation}
in which $A\in S^{-\infty }$ and each $a_\nu\in S^d_{1,1}$ with
$\hat a_\nu(\xi,\eta)=0$ for $2^{\nu+1}|\xi+\eta|<|\xi|$ and 
seminorms $C_{\alpha,\beta}(a_\nu)=\cal O(2^{-\nu N})$ for each $N>0$.

Indeed, when $\chi(\xi+\eta,\varepsilon\eta)\ne 0$ the triangle inequality
gives $|\xi+\eta|\le \varepsilon|\eta|\le
\varepsilon|\xi+\eta|+\varepsilon|\xi|$, whence 
$|\xi+\eta|(1-\varepsilon)/\varepsilon\le |\xi|$, so that for one thing
\begin{equation}
  \hat a_{\chi,\varepsilon}(x,\eta)=
  \chi(\xi+\eta,\varepsilon\eta)\hat A(x,\eta)+
  \sum_{2^{\nu+1}>(1-\varepsilon)/\varepsilon}
  \chi(\xi+\eta,\varepsilon\eta)\hat a_\nu(x,\eta).
\end{equation}
Secondly, for each seminorm $C_{\alpha,\beta}$ in
$S^d_{1,1}$ one has $C_{\alpha,\beta}(a_{\nu,\chi,\varepsilon})
\le \varepsilon^{-|\alpha|}C_{\alpha,\beta}(a_\nu)$ by 
Lemma~\ref{Heps-lem}, so by estimating the geometric series by its first
term, the above formula entails that
\begin{equation}
  C_{\alpha,\beta}(a_{\chi,\varepsilon})
  \le
  C_{\alpha,\beta}(A_{\chi,\varepsilon})
  +\sum_{\varepsilon 2^{\nu+1}>1-\varepsilon} 
    \frac{ C_{N+|\alpha|}}{\varepsilon^{|\alpha|} 2^{\nu (N+|\alpha|)}}
  \le   C_{\alpha,\beta}(A_{\chi,\varepsilon})+ \frac{c}{\varepsilon^{|\alpha|}}
  (\frac{2\varepsilon}{1-\varepsilon})^{N+|\alpha|}.    
\end{equation}
This gives the factor $\varepsilon^N$ in \eqref{orderN-cnd} for
$0<\varepsilon\le 1/2$. For $1/2<\varepsilon<1$ the series is 
$\cal O(\varepsilon^{-|\alpha|})$ because $2^{-\nu}\le 
1<2\varepsilon/(1-\varepsilon)$ for all $\nu$. However, 
$1\le (2\varepsilon)^{N+|\alpha|}$ for such $\varepsilon$,
so \eqref{orderN-cnd} will follow for all $\varepsilon\in
\,]0,1[\,$. (It is seen directly that $|A_{\chi,\varepsilon}(x,\eta)|
\le c\varepsilon^N(1+|\eta|)^d$ etc, for only the case $\varepsilon|\eta|\ge
1$ is non-trivial, and then $\varepsilon^{-N}\le (1+|\eta|)^N$ while $A\in
S^{-\infty }$.)

In the deduction of \eqref{aAnu-eq} one can use a 
Littlewood--Paley partition of unity, say $1=\sum_{\nu=0}^\infty \Phi_\nu$
with dilated functions $\Phi_\nu(\eta)=\Phi(2^{-\nu}\eta)\ne0$ 
only for $\tfrac{11}{20}2^{\nu}\le |\eta|\le 
\tfrac{13}{10}2^\nu$ if $\nu\ge 1$.
Beginning with a trivial split $a^*=A_0+A_1$ into two terms for which 
$A_0\in S^{-\infty}$ 
and $A_1\in S^d_{1,1}$ such that $A_1(x,\eta)=0$ for $|\eta|<1/2$, this gives
\begin{equation}
  \hat a{}^*(\xi,\eta)=\hat A_0(\xi,\eta)+\sum_{\nu=0}^\infty 
    \Phi_\nu(\xi/|\eta|)\hat A_1(\xi,\eta).
  \label{a*A-eq}
\end{equation}
This yields the desired $a_\nu(x,\eta)$ by taking the adjoint of
$\cal F^{-1}_{\xi\to x}(\Phi_\nu(\tfrac{\xi}{|\eta|})\hat A_1(\xi,\eta))$, that is, of the symbol
$\int |2^\nu\eta|^n\check \Phi(|2^\nu\eta|y)A_1(x-y,\eta)\,dy$.
Indeed, it follows directly from \cite[Prop.~3.3]{H88} (where the proof uses
Taylor expansion and vanishing moments of $\check \Phi$ for $\nu\ge 1$) that
$a^*_{\nu}$ belongs to
$S^d_{1,1}$ with $(2^{N\nu}a^*_\nu)_{\nu\in \N}$ bounded in
$S^d_{1,1}$ for all $N>0$. Therefore \eqref{a*A-eq} gives \eqref{aAnu-eq} by
inverse Fourier transformation. Moreover, 
$\hat a{}^*_\nu(\xi,\eta)$ is for $\nu\ge 1$ is supported by the region 
\begin{equation}
  \tfrac{11}{20}2^\nu|\eta|\le |\xi|\le \tfrac{13}{20}2^\nu |\eta|,
\end{equation}
where a fortiori $1+|\xi+\eta|\ge |\xi|-|\eta|\ge (\tfrac{11}{20}2^\nu-1)|\eta|\ge \tfrac{1}{10}|\eta|$,
so it is clear that
\begin{equation}
  \hat a{}^*_\nu(\xi,\eta)=0 \quad\text{ if }10(|\xi+\eta|+1)<|\eta|. 
\end{equation}
According to Lemma~\ref{GB-lem} this implies that $a_\nu=a^{**}_\nu$ is also in
$S^d_{1,1}$ and that, because of the above boundedness in $S^d_{1,1}$,
there is a constant $c$ independent of $\nu$ such that
\begin{equation}
  \begin{split}
  |D^\alpha_\eta D^\beta_x a_\nu(x,\eta)|
&\le 
  C_{\alpha,\beta}(a^*_\nu)10(1+10^{d-|\alpha|+|\beta|})
  (1+|\eta|)^{d-|\alpha|+|\beta|} 
\\
&\le 
  c2^{-N\nu}(1+|\eta|)^{d-|\alpha|+|\beta|}.
  \end{split}
\end{equation}
Therefore, the $a_\nu$ tend rapidly to $0$, which completes the proof of Theorem~\ref{a*-thm}.

\subsubsection{Consequences for the Self-Adjoint Subclass} \label{selfad-sssect}
One can set Theorem~\ref{a*-thm} in relation to the definition
by vanishing frequency modulation, simply by elaborating on the above proof:

\begin{cor}   \label{a*-cor}
On $\widetilde S^d_{1,1}(\Rn\times \Rn)$ 
the adjoint operation is stable with respect to vanishing
frequency modulation in the sense that, when $a\in \widetilde S^d_{1,1}$, 
$\psi\in C^\infty_0(\Rn)$ with $\psi=1$ around $0$, then
\begin{equation}
  \big(\psi(2^{-m}D_x)a(x,\eta)\psi(2^{-m}\eta) \big)^*
  \xrightarrow[m\to\infty ]{~}
  a(x,\eta)^*
\end{equation}
holds in the topology of $S^{d+1}_{1,1}(\Rn\times \Rn)$.
\end{cor}
\begin{proof}
For brevity $b_m(x,\eta)=\psi(2^{-m}D_x)a(x,\eta)\psi(2^{-m}\eta)$ denotes
the symbol that is frequency modulated in both variables.
The proof consists in insertion of $a(x,\eta)-b_m(x,\eta)$ into \eqref{a*-est},
where the first sum tends to $0$ for $m\to\infty $ by majorised convergence. 

Note that for each $\nu\ge 0$ in the first sum of \eqref{a*-est} one must control
$C_{\alpha,\beta}(e^m_{2^{-\nu}})$ for $m\to\infty $ and
\begin{equation}
   \hat e{}^m_{2^{-\nu}}(\xi,\eta)= (\chi(\xi+\eta,2^{-\nu}\eta)-\chi(\xi+\eta,2^{-\nu-1}\eta))
  (1-\psi(2^{-m}\xi)\psi(2^{-m}\eta))\hat a(\xi,\eta).
\end{equation}
To do so, a convolution estimate first gives
$p_{\alpha,\beta}(b_m)\le c \sum_{\gamma\le \alpha}p_{\gamma,\beta}(a)$,
whence $(b_m)_{m\in\N}$ is bounded in $S^d_{1,1}$.
Similar arguments yield that $b_m\to a$ in $S^{d+1}_{1,1}$ for
$m\to\infty $; cf.\ \cite[Lem.~2.1]{JJ08vfm}. 
Moreover, for each $\nu\ge 0$, every seminorm $p_{\alpha,\beta}$ now on $S^{d+1}_{1,1}$, gives
\begin{equation}
  p_{\alpha,\beta}(e^m_{2^{-\nu}})
  \le p_{\alpha,\beta}((a-b_m)_{_{\chi,2^{-\nu}}})
     +p_{\alpha,\beta}((a-b_m)_{\chi,2^{-\nu-1}}).
\end{equation}
Here both terms on the right-hand side tend to $0$ for $m\to\infty $, in
view of the continuity of $a\mapsto a_{\chi,\varepsilon}$ on
$S^{d+1}_{1,1}$; cf.\ Lemma~\ref{Heps-lem}. 
Hence $C_{\alpha,\beta}(e^{m}_{2^{-\nu}})\to0$ for $m\to\infty $.

It therefore suffices to replace $d$ by $d+1$ in \eqref{a*-est} and
majorise. However, 
$a\mapsto a_{\chi,\varepsilon}$ commutes with $a\mapsto b_m$ as maps
in $\cal S'(\Rn\times \Rn)$, so since $a\in \widetilde S^{d+1}_{1,1}$, it
follows from \eqref{orderN-cnd} that
\begin{equation}
  p_{\alpha,\beta}((a-b_m)_{\chi,\varepsilon})
\le 
  p_{\alpha,\beta}(a_{\chi,\varepsilon})
 +c\sum_{\gamma\le \alpha} p_{\gamma,\beta}(a_{\chi,\varepsilon})
  \le 
  (1+c)(\sum_{\gamma\le \alpha}C_{\gamma,\beta,N})
  \varepsilon^N
\le 
  C'_{\alpha,\beta,N}\varepsilon^N.
\end{equation}
Using this in the previous inequality, $C_{\alpha,\beta}(e^{m}_{2^{-\nu}})\le C2^{-\nu N}$ 
is obtained for $C$ independent of $m\in \N$.  Now it follows from \eqref{a*-est}
that $b_m(x,\eta)^*\to a(x,\eta)^*$ in $S^{d+1}_{1,1}$ as desired.
\end{proof}

Thus prepared, the proof of Proposition~\ref{GB-prop} can now be repeated
from \eqref{*conv-eq} onwards, which immediately gives the first main
result of the paper:

\begin{thm}   \label{tildeS-thm}
When a symbol $a(x,\eta)$ of type $1,1$ belongs to the class $\tilde
S^d_{1,1}(\Rn\times \Rn)$, as characterised in Theorem~\ref{a*-thm}, then
$a(x,D)$ is everywhere defined and continuous
\begin{equation} \label{aS'S'-eq}
  a(x,D)\colon \cal S'(\Rn)\to \cal S'(\Rn)
\end{equation}
It equals the adjoint of 
$\OP(e^{\im D_x\cdot D_\eta}\bar a(x,\eta))\colon\cal S\to\cal S$.
\end{thm}

Like for Proposition~\ref{GB-prop}, there seems to be no previous attempts in the literature to
obtain this clarification (Theorem~\ref{tildeS-thm} was stated without proof in \cite{JJ11lp}). 
However, it seems to be open whether \eqref{aS'S'-eq} conversely implies that $a\in\tilde S^d_{1,1}$.

\section{Dyadic Corona Decompositions}   \label{corona-sect}

This section adopts Littlewood--Paley techniques to provide a 
passage to auxiliary operators $a^{(j)}(x,D)$, $j=1,2,3$, 
which may be easily analysed with the pointwise estimates of Section~\ref{pe-sect}.

\subsection{The Paradifferential Splitting}
Recalling the definition of type $1,1$-operators in
\eqref{aPsi-eq} and \eqref{aPsi'-eq}, 
it is noted that to each modulation function $\psi$, 
i.e.\ $\psi\in C^\infty_0(\Rn)$ with
$\psi=1$ in a neighbourhood of $0$, there exist $R>r>0$ with $R\ge 1$ satisfying
\begin{equation}
  \psi(\xi)=1\quad\text{for}\quad |\xi|\le r;
  \qquad
  \psi(\xi)=0\quad\text{for}\quad |\xi|\ge R.
  \label{Rr-eq}
\end{equation}
For fixed $\psi$ it is convenient to 
take an integer $h\ge 2$ so large that $2R< r2^h$.

To obtain a Littlewood--Paley decomposition from $\psi$, set 
$\varphi=\psi-\psi(2\cdot )$. Then a dilation of this function is
supported in a corona,
\begin{equation}
  \supp \varphi(2^{-k}\cdot )\subset \bigl\{\,\xi \bigm| 
   r2^{k-1}\le|\xi|\le R2^k\,\bigr\},
  \qquad \text{for }k\ge 1.
  \label{phi-eq}
\end{equation}
The identity $1=\psi(x)+\sum_{k=1}^\infty \varphi(2^{-k}\xi)$ 
follows by letting $m\to\infty $ in the telescopic sum,
\begin{equation}
  \psi(2^{-m}\xi)=\psi(\xi)+\varphi(\xi/2)+\dots+\varphi(\xi/2^m).
  \label{tele-eq}
\end{equation}

Using this, functions $u(x)$ and symbols $a(x,\eta)$ will be localised to
frequencies $|\eta|\approx 2^j$ as
\begin{equation}
  u_j=\varphi(2^{-j}D)u, \qquad
  a_j(x,\eta)=\varphi(2^{-j}D_x)a(x,\eta).
  \label{ua_j-eq}
\end{equation}
Localisation to balls given by $|\eta|\le R2^j$ are written with upper indices,
\begin{equation}
  u^j=\psi(2^{-j}D)u, \qquad
  a^j(x,\eta)=\psi(2^{-j}D_x)a(x,\eta).
  \label{ua^j-eq}
\end{equation}

In addition $u_0=u^0$ and $a_0=a^0$; as an \emph{index convention} they are all
taken $\equiv0$ for $j<0$. (To avoid having two different meanings of sub-
and superscripts, the dilations $\psi(2^{-j}\cdot  )$ are written as such,
with the corresponding Fourier multiplier as $\psi(2^{-j}D)$, and similarly
for $\varphi$). Note that the corresponding operators are $a^k(x,D)=\OP(\psi(2^{-k}D_x)a(x,\eta))$ etc.

Inserting the relation \eqref{tele-eq} twice in \eqref{aPsi-eq}, bilinearity gives
\begin{equation}
  \OP(\psi(2^{-m}D_x)a(x,\eta)\psi(2^{-m}\eta))u
= \sum_{j,k=0}^m   a_j(x,D)u_k.
  \label{bilin-eq}
\end{equation}
Of course the sum may be split in three groups having $j\le k-h$,
$|j-k|<h$ and $k\le  j-h$. For $m\to\infty $ this yields the well-known paradifferential decomposition
\begin{equation}
  a_{\psi}(x,D)u=
  a_{\psi}^{(1)}(x,D)u+a_{\psi}^{(2)}(x,D)u+a_{\psi}^{(3)}(x,D)u,
  \label{a123-eq}
\end{equation}
whenever $a$ and $u$ fit together such that the three series below converge in $\cal D'(\Rn)$:
\begin{align}
    a_{\psi}^{(1)}(x,D)u&=\sum_{k=h}^\infty \sum_{j\le k-h} a_j(x,D)u_k
  =\sum_{k=h}^\infty a^{k-h}(x,D)u_k
  \label{a1-eq}\\
  a_{\psi}^{(2)}(x,D)u&= \sum_{k=0}^\infty
               \bigl(a_{k-h+1}(x,D)u_k+\dots+a_{k-1}(x,D)u_k+a_{k}(x,D)u_k
\notag\\[-2\jot]
   &\qquad\qquad
                +a_{k}(x,D)u_{k-1} +\dots+a_k(x,D)u_{k-h+1}\bigr) 
  \label{a2-eq}\\
   a_{\psi}^{(3)}(x,D)u&=\sum_{j=h}^\infty\sum_{k\le j-h}a_j(x,D)u_k
   =\sum_{j=h}^\infty a_j(x,D)u^{j-h}.
  \label{a3-eq}
\end{align}
Note the shorthand $a^{k-h}(x,D)$ for $\sum_{j\le k-h}a_j(x,D)=\op{OP}(\psi(2^{h-k}D_x)a(x,\eta))$ etc.
Using this and the index convention, the so-called symmetric term in \eqref{a2-eq} has the brief form
\begin{equation}
 a_{\psi}^{(2)}(x,D)u=\sum_{k=0}^\infty
((a^{k}-a^{k-h})(x,D)u_k+a_k(x,D)(u^{k-1}-u^{k-h})). 
\label{a2-eq'}
\end{equation}

In the following the subscript $\psi$ is usually dropped because this
auxiliary function will be fixed ($\psi$ was left out already in $a_j$ and
$a^j$; cf.\ \eqref{ua_j-eq}--\eqref{ua^j-eq}). 
Note also that the above $a^{(j)}(x,D)$ for now is just a
convenient notation for the infinite series. The full justification of
this operator notation will first result from Theorems~\ref{a123-thm}--\ref{a2a-thm} below.

\begin{rem}   \label{cutoff-rem}
It was tacitly used in \eqref{bilin-eq} and \eqref{a1-eq}--\eqref{a3-eq}
that one has
\begin{equation}
  a_j(x,D)u_k=\OP(a_j(x,\eta)\varphi(2^{-k}\eta))u.
  \label{ajuk-eq}
\end{equation}
This is because, with $\chi\in C^\infty_0$ equalling $1$
on $\supp \varphi(2^{-k}\cdot)$, both sides are equal to
\begin{equation}
  \OP(a_j(x,\eta)\chi(\eta))u_k.
\end{equation}
Indeed, while this is trivial for the right-hand side of \eqref{ajuk-eq},
where the symbol is in $S^{-\infty }$,
it is for the type $1,1$-operator on the left-hand side of \eqref{ajuk-eq} a fact that
follows at once from \eqref{aFE-eq}.
Thus the inclusion $\cal F^{-1}\cal E'\subset D(a(x,D))$ in
\eqref{aFE-eq} is crucial for the simple formulae in the present paper.
Analogously Definition~\ref{a11-defn} may be rewritten briefly as $a(x,D)u=\lim_m a^m(x,D)u^m$.
\end{rem}

The importance of the decomposition in \eqref{a1-eq}--\eqref{a3-eq} 
lies in the fact that the summands have localised spectra. E.g.\ there is a dyadic corona property:

\begin{prop}  \label{corona-prop}
If $a\in S^d_{1,1}(\Rn\times \Rn)$ and $u\in \cal S'(\Rn)$, and 
$r$, $R$ are chosen as in \eqref{Rr-eq} for each auxiliary function $\psi$,
then every $h\in \N$ such that $2R< r2^h$ gives
\begin{align}
  \supp\cal F(a^{k-h}(x,D)u_k)&\subset
  \bigl\{\,\xi\bigm| 
  R_h2^k\le|\xi|\le \frac{5R}{4} 2^k\,\bigr\}
  \label{supp1-eq}  \\
  \supp\cal F(a_k(x,D)u^{k-h})&\subset
  \bigl\{\,\xi \bigm| 
  R_h2^k\le|\xi|\le \frac{5R}{4} 2^k\,\bigr\},
  \label{supp3-eq}
\end{align}
whereby $R_h=\tfrac{r}{2}-R2^{-h}>0$.
\end{prop}
\begin{proof}
By \eqref{phi-eq} and the Spectral Support Rule, cf.\ the last part of Theorem~\ref{supp-thm},
\begin{equation}
  \supp\cal F(a^{k-h}(x,D)u_k)\subset
  \bigl\{\,\xi+\eta \bigm| (\xi,\eta)\in \supp
    (\psi_{h-k}\otimes1)\hat a, 
\
   r2^{k-1}\le |\eta|\le R2^k \,\bigr\}.
\end{equation}
So by the triangle inequality every $\zeta=\xi+\eta$ in the support fulfils,
as $h\ge 2$,
\begin{equation}
  r2^{k-1}-R2^{k-h}
 \le|\zeta|\le R2^{k-h}+R2^k\le \tfrac{5}{4}R2^k.
\end{equation}
This shows \eqref{supp1-eq}, and \eqref{supp3-eq} follows analogously.
\end{proof}

To achieve simpler constants one could take $h$ so large that
$4R\le r2^h$, which instead of $R_h$ would allow $r/4$ (and $9R/8$). 
But the present choice of $h$ is preferred in order to reduce the number
of terms in $a^{(2)}(x,D)u$.

In comparison the terms in $a^{(2)}(x,D)u$ only satisfy a dyadic ball
condition. Previously this was observed e.g.\ for functions $u\in \bigcup H^s$
in \cite{JJ05DTL}, as was the fact that when
the twisted diagonal condition \eqref{tdc-cnd} holds, 
then the situation improves for large $k$. This is true for arbitrary $u$:

\begin{prop}   \label{ball-prop}
When $a\in S^d_{1,1}(\Rn\times \Rn)$, $u\in \cal S'(\Rn)$, and 
$r$, $R$ are chosen as in \eqref{Rr-eq} for each auxiliary function $\psi$,
then every $h\in \N$ such that $2R< r2^h$ gives
\begin{equation}
  \begin{split}
  \supp\cal F\big(a_k(x,D)(u^{k-1}-u^{k-h})\big)
  &\bigcup
  \supp\cal F\big((a^k-a^{k-h})(x,D)u_k\big)
\\
  \subset
  \bigl\{\,\xi\in\Rn &\bigm| |\xi|\le 2R 2^k\,\bigr\}  
  \end{split}
  \label{supp2-eq}
\end{equation}
If \eqref{tdc-cnd} holds for some $B\ge 1$, 
then the support is contained in the annulus
\begin{equation}
  \bigl\{\,\xi \bigm| \frac{r}{2^{h+1}B} 2^k \le |\xi|\le 2R2^k\,\bigr\}
\quad\text{for all $k\ge h+1+\log_2(\frac Br)$.}
  \label{supp2'-eq}
\end{equation}
\end{prop}

\begin{proof}
As in Proposition~\ref{corona-prop},
$\supp\cal Fa_k(x,D)(u^{k-1}-u^{k-h})$ is seen to be a subset of
\begin{equation}
  \bigl\{\,\xi+\eta \bigm| (\xi,\eta)\in \supp(\varphi_k\otimes 1)
\hat a,\ r2^{k-h}\le |\eta|\le R2^{k-1}\,\bigr\}.
\end{equation}
Thence any $\zeta$ in the support fulfils $|\zeta|\le R2^k+R2^{k-1}=(3R/2)2^k$.
If \eqref{tdc-cnd} holds, then one has $B(1+|\xi+\eta|)\ge |\eta|$ on
$\supp\cal F_{x\to\xi}a$, so for all $k$ larger than the given limit
\begin{equation}
  |\zeta|\ge \tfrac{1}{B}|\eta|-1\ge \tfrac{1}{B}r2^{k-h}-1\ge 
  (\tfrac{r}{2^hB}-2^{-k})2^k\ge \tfrac{r}{2^{h+1}B}2^k.
\end{equation}

The term $(a^k-a^{k-h})(x,D)u_k$ is analogous but will cause $3R/2$ to be replaced by $2R$.
\end{proof}

\begin{rem}
  \label{a123-rem}
The inclusions in Propositions~\ref{corona-prop} and \ref{ball-prop} have been a main reason for
the introduction of the paradifferential splitting \eqref{a123-eq}
in the 1980's, but they were then only derived for elementary symbols; cf.\
\cite{Bou83,Bou88,Y1}. With the Spectral Support Rule, cf.\ Theorem~\ref{supp-thm}, 
this restriction is redundant; cf.\ also the remarks to \eqref{sFAu-eq} in the introduction.
\end{rem}

\subsection{Polynomial Bounds}   
  \label{poly-ssect}
In the treatment of $a^{(1)}(x,D)u$ and $a^{(3)}(x,D)u$ 
in \eqref{a1-eq} and \eqref{a3-eq}
one may conveniently commence by observing that, according to
Proposition~\ref{corona-prop}, the terms in
these series fulfil
condition \eqref{DAC-cnd} in Lemma~\ref{corona-lem} for
$\theta_0=\theta_1=1$.

So to deduce their convergence from Lemma~\ref{corona-lem},
it remains to obtain the polynomial bounds in \eqref{CM-cnd}.
For this it is natural to use the efficacy of the pointwise estimates
in Section~\ref{pe-sect}:

\begin{prop}   \label{pe13-prop}
If $a(x,\eta)$ is in $S^d_{1,1}(\Rn\times \Rn)$ and 
$N\ge \order_{\cal S'}(\cal Fu)$ fulfils $d+N\ne0$, then 
\begin{align}
  |a^{k-h}(x,D)u_k(x)|&\le  c 2^{k(N+d)}(1+|x|)^N,
  \label{a1pe-eq}  \\
  |a_k(x,D)u^{k-h}(x)|&\le c 2^{k(N+d)_+}(1+|x|)^N,
  \label{a3pe-eq}\\
  |(a^{k}-a^{k-h})(x,D)u_k(x)| &\le c2^{k(N+d)}(1+|x|)^{N},
  \label{a2pe'-eq}\\
  |a_k(x,D)(u^{k-1}-u^{k-h})(x)|
  &\le c2^{k(N+d)}(1+|x|)^{N}.
  \label{a2pe-eq}
\end{align}
\end{prop}
\begin{proof}
The second inequality follows by taking 
the two cut-off functions in Proposition~\ref{cutoff-prop} as
$\Phi=\varphi$ and $\Psi=\psi(2^{-h}\cdot )$.
The first claim is seen by interchanging
their roles, i.e.\ for $\Phi=\psi(2^{-h}\cdot)$ and $\Psi=\varphi$;
the latter is $0$ around the origin so $N+d$ is obtained without the
positive part.

Clearly similar estimates hold for the terms in $a^{(2)}(x,D)u$. E.g.,
taking $\psi-\psi(2^{-h}\cdot )$ and $\varphi$, respectively, as the cut-off
functions in Proposition~\ref{cutoff-prop}, one finds for $k\ge h$
the estimate in \eqref{a2pe'-eq}. Note that the positive part can be
avoided for $0\le k< h$ by using a sufficiently large constant.
\end{proof}

The difference in the above estimates appears because $u_k$ in
\eqref{a1pe-eq} has spectrum in a corona.
However, one should not confound this with spectral inclusions like
\eqref{DAC-cnd} that one might obtain after application of 
$a^{k-h}(x,D)$, for these are
irrelevant for the pointwise estimates here.

\subsection{Induced Paradifferential Operators} \label{para-ssect}
Although 
\eqref{a1-eq}--\eqref{a3-eq} yield a well-known splitting,
the operator notation $a^{(j)}(x,D)$ requires justification
in case of type $1,1$-operators.

Departing from the right hand sides of \eqref{a1-eq}--\eqref{a3-eq} one is
via \eqref{ajuk-eq} led directly to the symbols
\begin{align}
    a^{(1)}(x,\eta)&=\sum_{k=h}^\infty a^{k-h}(x,\eta)\varphi(2^{-k}\eta)
  \label{a1'-eq}
\\
   a^{(3)}(x,\eta)& =\sum_{j=h}^\infty a_j(x,\eta)\psi(2^{-(j-h)}\eta).
  \label{a3'-eq}
\end{align}
In addition, letting
$\delta_{k\ge h}$ stand for $1$ when $k\ge h$ and for $0$ in case $k<h$,
\begin{align}
  a^{(2)}(x,\eta)&=\sum_{k=1}^\infty
        \big((a^{k}(x,\eta)-a^{k-h}(x,\eta))\varphi(2^{-k}\eta)
\notag
\\
  &\qquad
  +a_k(x,\eta)(\psi(2^{-(k-1)}\eta)-\psi(2^{-(k-h)}\eta) \delta_{k\ge h})\big)
     + a^{0}(x,\eta)\psi(\eta)
  \label{a2'-eq}
\end{align}
These three series converge in the Fr\'echet space $S^{d+1}_{1,1}(\Rn\times
\Rn)$, for the sums are locally finite. Therefore it is clear that
\begin{equation}
  a(x,\eta)=a^{(1)}(x,\eta)+a^{(2)}(x,\eta)+a^{(3)}(x,\eta),
\label{a1a2a3-id}
\end{equation}
where some of the partially Fourier transformed symbols have conical supports, 
\begin{equation}
  \hat a{}^{(1)}(\xi,\eta)\ne0 \implies |\xi|\le\tfrac{2R}{r2^h}|\eta|,
  \qquad
  \hat a{}^{(3)}(\xi,\eta)\ne0 \implies |\eta|\le\tfrac{2R}{r2^h}|\xi|.
\label{hata-supp} 
\end{equation}
This well-known fact follows from the supports of $\psi$ and $\varphi$.
But a sharper exploitation gives

\begin{prop}   \label{a13tdc-prop}
For each $a\in S^d_{1,1}$ and every modulation function
$\psi\in C^\infty_0(\Rn)$ in \eqref{Rr-eq}, the associated symbols 
$a^{(1)}_{\psi}(x,\eta)$ and $a^{(3)}_{\psi}(x,\eta)$ fulfil the twisted diagonal condition
\eqref{tdc-cnd}. 
\end{prop}
\begin{proof}
When $\hat a{}^{(3)}(\xi,\eta)\ne0$ it follows from \eqref{hata-supp}, which in particular
yields $|\eta|<|\xi|$, that
\begin{equation}
  |\xi+\eta|\ge |\xi|-|\eta|\ge |\xi|(1-\tfrac{2R}{r2^h})>|\eta|(1-\tfrac{2R}{r2^h}).
\label{xieta-lb}
\end{equation}
Therefore $\hat a{}^{(3)}(\xi,\eta)=0$ whenever $B_1|\xi+\eta|<|\eta|$
holds for $B_1=(1-\frac{2R}{r2^h})^{-1}$; 
a fortiori \eqref{tdc-cnd} is fulfilled  with $B=B_1>1$. 
The case of $a^{(1)}$ is a little simpler.
\end{proof}

To elucidate the role of the twisted diagonal, note that
the lower bound in Proposition~\ref{corona-prop} reappears
by using $|\xi|\ge r2^{k-1}$ in the middle of \eqref{xieta-lb}.

Anyhow, it is a natural programme to verify that $u\in\cal S'$ belongs to the domain
of the operator $a^{(j)}(x,D)$ precisely when the previously introduced series denoted
$a^{(j)}(x,D)u$ converges; cf.\ \eqref{a1-eq}--\eqref{a3-eq}. In view of the definition by
vanishing frequency modulation in \eqref{aPsi-eq} ff, this will
necessarily be lengthy because a second modulation function $\Psi$ has
to be introduced.

To indicate the details for $a^{(1)}(x,\eta)$, let $\psi, \Psi\in C^\infty_0(\Rn)$ be
equal to $1$ around the origin, and let $\psi$ be used as the fixed
modulation function entering $a^{(1)}(x,D)=a^{(1)}_{\psi}(x,D)$ in \eqref{a1-eq};
and set $\varphi=\psi-\psi(2\cdot )$. The numbers $r, R$ and $h$ are 
then chosen in relation to $\psi$ as in \eqref{Rr-eq}.

Moreover, $\Psi$ is used for the frequency modulation made when
Definition~\ref{a11-defn} is applied to $a^{(1)}_\psi(x,D)$.
This gives the following identity in $S^d_{1,1}$,
where prime indicates a finite sum,
\begin{multline}
  \Psi(2^{-m}D_x)a^{(1)}(x,\eta)\Psi(2^{-m}\eta)=
   \sum_{k=h}^{m+\mu}a^{k-h}(x,\eta)\varphi(2^{-k}\eta)
\\
   + \sideset{}{'}\sum_k\Psi(2^{-m}D_x)a^{k-h}(x,\eta)
  \varphi(2^{-k}\eta)\Psi(2^{-m}\eta).
\label{a1u'-eq}
\end{multline}
Indeed, if $\lambda,\Lambda>0$ fulfil that
$\Psi(\eta)=1$ for $|\eta|\le \lambda$ while $\Psi=0$ for $|\eta|\ge \Lambda$, 
the support of $\varphi(2^{-k}\eta )$ in \eqref{a1'-eq}
lies by \eqref{phi-eq} in one of the `harmless' level sets
$\Psi(2^{-m}\eta)=1$ or $\Psi(2^{-m}\eta)=0$ when, respectively,
\begin{equation}
  R 2^k\le   \lambda 2^m\quad\text{or}\quad r 2^{k-1}\ge  \Lambda 2^m.
\end{equation}
That is, $\supp\varphi(2^{-k}\cdot)$ is contained in these level sets unless $k$ fulfils
\begin{equation}
  m+\log_2(\lambda/R)<k< m+1 +\log_2(\Lambda/r).
\end{equation}
Therefore the primed sum has at most
$1+\log_2\tfrac{R\Lambda}{r\lambda}$ terms, independently of the 
parameter $m$; in addition $\Psi(2^{-m}\eta)$ and $\Psi(2^{-m}D_x)$ disappear 
from the other terms, as stated in \eqref{a1u'-eq}.

Consequently, with $\mu= [\log_2(\lambda/R)]$ and $k=m+l$, for
$l\in\Z$, one has for $u\in \cal S'(\Rn)$ that
\begin{multline}
  \OP(\Psi(2^{-m}D_x)a^{(1)}(x,\eta)\Psi(2^{-m}\eta))u=
   \sum_{k=h}^{m+\mu}a^{k-h}(x,D)u_k
\\
   + \sideset{}{'}\sum_{\mu<l<1+\log_2(\Lambda/r)}
     \OP(\Psi(2^{-m}D_x)\psi(2^{h-l-m}D_x)a(x,\eta)
     \varphi(2^{-m-l}\eta)\Psi(2^{-m}\eta))u.
  \label{a1u-eq}
\end{multline}

A similar reasoning applies to $a^{(3)}(x,\eta)$. The main difference is that
the possible inclusion of $\supp\varphi(2^{-j}\cdot)$, into the level sets
where $\Psi(2^{-m}\cdot )$ equals $1$ or $0$, in this case applies to the
symbol 
$\Psi(2^{-m}D_x)a_j(x,\eta)=\cal F^{-1}_{\xi\to
x}(\Psi(2^{-m}\xi)\varphi(2^{-j}\xi)\hat a(\xi,\eta))$. Therefore one has for the same~$\mu$,
\begin{multline}
  \OP(\Psi(2^{-m}D_x)a^{(3)}(x,\eta)\Psi(2^{-m}\eta))u=
   \sum_{j=h}^{m+\mu}a_j(x,D)u^{j-h}
\\
   + \sideset{}{'}\sum_{\mu<l<1+\log_2(\Lambda/r)}
     \OP(\Psi(2^{-m}D_x)\varphi(2^{-l-m}D_x)a(x,\eta)
     \psi(2^{h-m-l}\eta)\Psi(2^{-m}\eta))u.
  \label{a3u-eq}
\end{multline}

Treating $a^{(2)}_{\psi}(x,D)$ analogously, it is not difficult
to see that once again the central issue is whether $\supp\varphi(2^{-k}\cdot)$
is contained in  the set where $\Psi(2^{-m}\cdot )=1$ or $=0$. So when
$m\ge h$ for simplicity, one has for the
same~$\mu$, and with primed sums over the same integers $l$ as above,
\begin{equation}
\begin{split}
  \OP(\Psi(2^{-m}D_x)a^{(2)}(x,&\eta)\Psi(2^{-m}\eta))u =
   \sum^{m+\mu}_{k=0} \big((a^k-a^{k-h})(x,D)u_k +a_k(x,D)(u^{k-1}-u^{k-h})\big)
\\
   &+ \sideset{}{'}\sum
     \OP(\Psi(2^{-m}D_x)(a^{m+l}(x,\eta)-a^{m+l-h}(x,\eta))
     \varphi(2^{-m-l}\eta)\Psi(2^{-m}\eta))u
\\
   &+ \sideset{}{'}\sum
     \OP(\Psi(2^{-m}D_x)a_{m+l}(x,\eta)
     (\psi(2^{1-m-l}\eta)-\psi(2^{h-m-l}\eta))\Psi(2^{-m}\eta))u.
  \label{a2u-eq}
\end{split}
\end{equation}

The programme introduced after Proposition~\ref{a13tdc-prop} is now completed by
letting $m\to\infty $ in \eqref{a1u-eq}--\eqref{a2u-eq}
and observing that the infinite series in \eqref{a1-eq}--\eqref{a3-eq} reappear in this
way. Of course, this relies on the fact that the remainders in the 
primed sums over $l$ can be safely ignored:

\begin{prop}
  \label{remainder-prop}
When $a(x,\eta)$ is given in $S^d_{1,1}(\Rn\times \Rn)$ and $\Psi$, $\psi\in
C^\infty_0(\Rn)$ equal $1$ in neighbourhoods of the origin, then
it holds for every $u\in \cal S'(\Rn)$ that each term (with $l$ fixed) in 
the primed sums in \eqref{a1u-eq}--\eqref{a3u-eq} tends to $0$ in 
$\cal S'(\Rn)$ for $m\to\infty $.

This is valid for \eqref{a2u-eq} too,
if $a(x,\eta)$ in addition fulfils the twisted diagonal condition
\eqref{tdc-cnd}.
\end{prop}

\begin{proof}
To show that each remainder term tends to $0$ for $m\to\infty $
and fixed $l$,
it suffices to verify \eqref{DAC-cnd} and \eqref{CM-cnd}
in view of Remark~\ref{corona-rem}.

For $a^{(1)}_{\psi}(x,D)$, note that by repeating the proof of
Proposition~\ref{corona-prop} (ignoring $\Psi$)
each remainder in \eqref{a1u-eq} has $\xi$ in its spectrum only
when $(R_h2^l) 2^m\le|\xi|\le \tfrac{5\cdot  2^l}{4}R2^m$.

Moreover, each remainder term is $\le c 2^{k(N+d)}(1+|x|)^{N}$ for
$N\ge \order_{\cal S'}(\hat u)$ according to
Proposition~\ref{cutoff-prop}, for with the cut-off functions 
$\Psi\psi(2^{h-l}\cdot )$ and $\varphi(2^{-l}\cdot )\Psi$ the latter is
$0$ around the origin. So a crude estimate by
$c2^{k(N+d_+)}(1+|x|)^{N+d_+}$ shows that \eqref{CM-cnd} is fulfilled.

Similarly for the primed sum in
\eqref{a3u-eq}, where $\psi(2^{h-l}\cdot )\Psi$ is $1$ around the origin;
which again results in the bound $c2^{k(N+d_+)}(1+|x|)^{N+d_+}$ for $N\ne-d$. 

The procedure also works for \eqref{a2u-eq},
for \eqref{DAC-cnd} is verified as in
Proposition~\ref{ball-prop}, cf.\ \eqref{supp2'-eq}, 
because the extra spectral localisations provided by 
$\Psi(2^{-m}\cdot )$ cannot increase the spectra.
For the pointwise estimates one may now use e.g.\ 
$\Psi\varphi(2^{-l}\cdot )$ and $(\psi(2^{1-l}\cdot )-\psi(2^{h-l}\cdot ))\Psi$
as the cut-off functions in the last part of \eqref{a2u-eq}.
This yields the proof of Proposition~\ref{remainder-prop}.
\end{proof}

An extension of the proposition's remainder analysis to general
$a^{(2)}_\psi(x,D)$ without a condition on the behaviour along the
twisted diagonal does not seem feasible.
But such results will follow in Section~\ref{split-sect}
from a much deeper investigation of $a(x,D)$ itself; cf.\ Theorem~\ref{a2a-thm}.

\begin{rem}
The type $1,1$-operator $a^{(1)}(x,D)$ induced by \eqref{a1'-eq} is a
\emph{paradifferential} operator in the sense of Bony~\cite{Bon}, as well as in H{\"o}rmander's
framework of residue classes in \cite[Ch.~10]{H97}. 
The latter follows from \eqref{hata-supp}, but will not be pursued here.
$a^{(2)}(x,D)$ and $a^{(3)}(x,D)$ are also called paradifferential operators, following Yamazaki~\cite{Y1}.
The decomposition \eqref{a123-eq}--\eqref{a3-eq} can be traced back to Kumano-go and Nagase,
who used a variant of $a^{(1)}(x,\eta)$ to smooth non-regular symbols, cf.\ \cite[Thm~1.1]{KuNa78}.
It was exploited in
continuity analysis of pseudo-differential operators in e.g.\ \cite{Y1,Mar91,JJ05DTL,Lan06}.
\end{rem}

\begin{rem}   \label{piprod-rem}
For pointwise multiplication decompositions analogous to \eqref{a123-eq} 
were used implicitly by Peetre \cite{Pee}, Triebel \cite{T-pmlt}; and more explicitly in the
paraproducts of Bony~\cite{Bon}.
Moreover, for $a=a(x)$ Definition~\ref{a11-defn} reduces to the product 
$\pi(a,u)$ introduced formally in \cite{JJ94mlt} as
\begin{equation}
  \pi(a,u)=\lim_{m\to\infty} a^m\cdot u^m.
  \label{pprod-eq}
\end{equation}
This was analysed in \cite{JJ94mlt}, including 
continuity properties deduced from \eqref{a123-eq}, that essentially
is a splitting of the generalised pointwise product
$\pi(\cdot,\cdot)$ into paraproducts. Partial associativity, i.e.\ 
$f\pi(a,u)=\pi(fu,a)=\pi(a,fu)$ for $f\in C^\infty$, was first
obtained with the refined methods developed later in
\cite[Thm.~6.7]{JJ08vfm}, though.
\end{rem}

\section{Action on Temperate Distributions}
  \label{split-sect}

\subsection{Littlewood--Paley Analysis of Type $\mathbf{1},\mathbf{1}$-Operators}   
   \label{lwp11-ssect}
First the full set of conclusions is drawn for the operators 
$a^{(j)}(x,D)$, $j=1,2,3$ studied in Section~\ref{para-ssect}. 
Of course none of them have anomalies if $a(x,\eta)$ fulfils \eqref{tdc-cnd}:

\begin{thm}
  \label{a123-thm}
When $a(x,\eta)$ is a symbol in $S^d_{1,1}(\Rn\times\Rn)$ for some $d\in \R$
and $\psi\in C^\infty_0(\Rn)$ equals $1$ around the origin,
then the associated type $1,1$-operators  
$a^{(1)}_\psi(x,D)$ and $a^{(3)}_\psi(x,D)$ 
are everywhere defined continuous linear maps
\begin{equation}
 a^{(1)}_\psi(x,D),\ a^{(3)}_\psi(x,D) 
  \colon \cal S'(\Rn)\to\cal S'(\Rn),
\end{equation}
that are given by formulae 
\eqref{a1-eq} and \eqref{a3-eq}, where the infinite series
converge rapidly in $\cal S'(\Rn)$ for every $u\in \cal S'(\Rn)$. The
adjoints are also in $\OP(S^d_{1,1}(\Rn\times \Rn))$.

If furthermore $a(x,\eta)$ fulfils the twisted diagonal condition
\eqref{tdc-cnd}, these conclusions are valid verbatim for the operator 
$a^{(2)}_\psi(x,D)$ given by the series in \eqref{a2-eq}. 
\end{thm}

\begin{proof}
As the symbols $a^{(1)}(x,\eta)$ and
$a^{(3)}(x,\eta)$ both belong to $S^d_{1,1}$ and fulfil 
\eqref{tdc-cnd} by Proposition~\ref{a13tdc-prop},
the corresponding  operators are 
defined and continuous on $\cal S'$ 
by Proposition~\ref{GB-prop}, with $a^{(1)}(x,D)^*$ and
$a^{(3)}(x,D)^{*}$ both of type $1,1$.

Since $\supp \cal F_{x\to\xi}a^{(2)}\subset \supp \cal F_{x\to\xi}a$
it follows that $a^{(2)}(x,D)$ satisfies \eqref{tdc-cnd}, when $a(x,\eta)$ does so.
Hence the preceding argument also applies to $a^{(2)}(x,D)$, so that it is continuous on $\cal
S'$ with its adjoint being of type $1,1$.

Moreover, the series $\sum_{k=0}^\infty a^{k-h}(x,D)u_k$ in \eqref{a1-eq}
converges rapidly in $\cal S'$ for every $u\in \cal S'$.
This follows from $1^\circ$ of Lemma~\ref{corona-lem},
for the terms fulfil \eqref{DAC-cnd} and \eqref{CM-cnd}
by Proposition~\ref{corona-prop}, cf.\ \eqref{supp1-eq}, 
and Proposition~\ref{pe13-prop}, respectively.
(The latter gives a bound by $2^{k(N+d_+)}(1+|x|)^{N+d_+}$.)
Now the distribution $\sum_{k=0}^\infty a^{k-h}(x,D)u_k$ equals $a^{(1)}(x,D)u$ because of 
formula \eqref{a1u-eq}, since the primed sum there goes to $0$ for $m\to\infty$, 
as shown in Proposition~\ref{remainder-prop}.

Similarly Lemma~\ref{corona-lem} yields convergence of 
the series \eqref{a3-eq} for $a^{(3)}(x,D)u$ when $u\in \cal S'$.
By Propositions~\ref{ball-prop} and \ref{pe13-prop},
convergence of the $a^{(2)}$-series in \eqref{a2-eq'}
also follows from Lemma~\ref{corona-lem}.
The series identify with the operators 
in view of the remark made prior to Proposition~\ref{remainder-prop}.
\end{proof}

It should be emphasized that 
duality methods and pointwise estimates contribute in two different ways in
Theorem~\ref{a123-thm}: 
once the symbol $a^{(1)}(x,\eta)$ has been introduced,
continuity on $\cal S'(\Rn)$ of the associated type $1,1$-operator 
$a^{(1)}(x,D)$
is obtained by duality through Proposition~\ref{GB-prop}. However,  the
pointwise estimates in Section~\ref{pe-sect} yield 
(vanishing of the remainder terms, hence)
the identification
of $a^{(1)}(x,D)u$ with the series in \eqref{a1-eq}. 
Furthermore, the pointwise estimates also
give an explicit proof of the fact that $a^{(1)}(x,D)$ is defined on the
entire $\cal S'(\Rn)$, for the right-hand side of \eqref{a1-eq} does not
depend on the modulation function $\Psi$.
Similar remarks apply to $a^{(3)}(x,D)$. Thus duality methods and
pointwise estimates together lead to a deeper 
analysis of type $1,1$-operators.

\begin{rem}
Theorem~\ref{a123-thm} and its proof generalise a result of 
Coifman and Meyer \cite[Ch.~15]{MeCo97} in three ways.
They stated Lemma~\ref{corona-lem} for $\theta_0=\theta_1=1$ 
and derived a corresponding fact for 
paramultiplication, though only 
with a treatment of the first and third term.
\end{rem}

Going back to the given $a(x,D)$, one derives from
Theorem~\ref{a123-thm} and  \eqref{a123-eq} the following

\begin{thm} \label{tdc-thm}
When $a\in S^d_{1,1}(\Rn\times \Rn)$ fulfils  the twisted diagonal condition
\eqref{tdc-cnd}, then the associated
type $1,1$-operator  
$a(x,D)$ defined by vanishing frequency modulation
is an everywhere defined continuous linear map 
\begin{equation}
 a(x,D)\colon \cal S'(\Rn)\to\cal S'(\Rn),  
\end{equation}
with its adjoint $a(x,D)^*$ also in $\OP(S^d_{1,1}(\Rn\times \Rn))$.
The operator fulfils 
\begin{equation}
  a(x,D)u=a^{(1)}_\psi(x,D)u+a^{(2)}_\psi(x,D)u+a^{(3)}_\psi(x,D)u  
\end{equation}
for every $\psi\in C^\infty_0(\Rn)$ equal to $1$ in a
neighbourhood of the origin,
and the series in \eqref{a1-eq}--\eqref{a3-eq} 
converge rapidly in $\cal S'(\Rn)$ for every $u\in \cal S'(\Rn)$.
\end{thm}

To extend the discussion to general $a(x,D)$ without vanishing along the twisted diagonal, note that
Theorem~\ref{a123-thm} at least shows that  $a^{(1)}(x,D)u$ and $a^{(3)}(x,D)u$
are always defined and that \eqref{a1-eq} and~\eqref{a3-eq} are operator identities.

It remains to justify the operator notation $a^{(2)}(x,D)$ in \eqref{a2-eq} and to give its
precise relation to $a(x,D)$ itself. The point of departure is of
course the symbol splitting \eqref{a1a2a3-id};
the corresponding type $1,1$-operators are still denoted by $a^{(j)}(x,D)$. However, to
avoid ambiguity the series in \eqref{a1-eq}--\eqref{a3-eq} will now be temporarily 
written as $A^{(j)}_\psi u$, whence \eqref{a123-eq} amounts to
\begin{equation}
  a(x,D)u=A^{(1)}_\psi u+A^{(2)}_\psi u+A^{(3)}_\psi u \quad\text{ for $u\in\cal S'$}.
\label{a123tilde-id}
\end{equation}
Here the left-hand side exists if and only if the series
$A^{(2)}_\psi u$ converges, as $A^{(1)}_\psi u$, $A^{(3)}_\psi u$ always
converge by Theorem~\ref{a123-thm}.
This strongly indicates that \eqref{a2a-dom} below is true.
In fact, this main result of the analysis is obtained by frequency modulation:

\begin{thm}
  \label{a2a-thm}
  When $a(x,\eta)\in S^d_{1,1}(\Rn\times\Rn)$ and $a^{(2)}(x,\eta)$ denotes
  the type $1,1$-symbol in \eqref{a2'-eq}, derived from the paradifferential decomposition
  \eqref{a123-eq}, then
  \begin{equation}
    D(a^{(2)}(x,D))=    D(a(x,D))
\label{a2a-dom}
  \end{equation}
and $u\in\cal S'(\Rn)$ belongs to these  domains if and only if
the series \eqref{a2-eq}, or equivalently \eqref{a2-eq'},  converges
in $\cal D'(\Rn)$\,---\,in which case (also) formulae \eqref{a2-eq},
\eqref{a2-eq'} are operator identities. 
\end{thm}

\begin{proof}
A variant of \eqref{a123tilde-id} follows at once from \eqref{a1a2a3-id},
using a second modulation function $\Psi$ and the brief
notation from Remark~\ref{cutoff-rem},
\begin{equation}
  a^m(x,D)u^m=\sum_{l=1,2,3} \OP(\Psi(2^{-m}D_x)a^{(l)}(x,\eta)\Psi(2^{-m}\eta))u.
\end{equation}
Here the terms with $l=1$ and $l=3$ always have $\Psi$-independent limits for 
$m\to\infty$ according to Theorem~\ref{a123-thm}, so it is clear from the definition by
vanishing frequency modulation that $u\in D^{(2)}(a(x,D))$ is equivalent to 
$u\in D(a(x,D))$, hence to convergence of $A^{(2)}_\Psi u$, cf.\ \eqref{a123tilde-id} ff.

As for the last claim,  whenever $u\in D(a(x,D))$, then passage to the limit ($m\to\infty$) in the above
equation yields the following, when \eqref{a1u-eq}, \eqref{a2u-eq} and \eqref{a3u-eq}
are applied, now with the remainders
in the primed sums there denoted by $R^{(1)}_mu$, $R^{(2)}_mu$, $R^{(3)}_mu$ for brevity:
\begin{equation}
  \begin{split}
    a(x,D)u&=\lim_{m\to\infty} \sum_{l=1,2,3}
    \OP(\Psi(2^{-m}D_x)a^{(l)}(x,\eta)\Psi(2^{-m}\eta))u
\\
  & =A^{(1)}_\psi u+A^{(2)}_\psi u+A^{(3)}_\psi u +0+\lim_{m\to\infty} R^{(2)}_mu +0.
  \end{split}
\label{6term-sum}
\end{equation}
Note that convergence of $R^{(2)}_m$ follows from that of the other
six terms; cf.\ Proposition~\ref{remainder-prop}.
Compared to \eqref{a123tilde-id} this yields $\lim_m R^{(2)}_m=0$,
which via \eqref{a2u-eq} gives that $a^{(2)}(x,D)u=A^{(2)}_\psi u$.
\end{proof}

\subsection{The Twisted Diagonal Condition of Arbitrary Order}
  \label{twist-ssect}

When $a(x,\eta)$ is in the self-adjoint subclass $\widetilde
S^d_{1,1}$, then it follows Theorem~\ref{tildeS-thm} that the
domains in \eqref{a2a-dom} equal $\cal S'$. 

However, it is interesting to give an explicit proof that the domains in \eqref{a2a-dom} equal $\cal S'$ whenever
$a\in\widetilde S^d_{1,1}$. This can be done in a natural way by extending the proof of
Theorem~\ref{a2a-thm}, where the special estimates in
\eqref{Hsigma-eq} enter the convergence proof for $a^{(2)}(x,D)u$
directly, because they are rather close to the symbol factors from the
factorisation inequalities in Section~\ref{pe-sect}. 
The full generality with $\theta_0<\theta_1$ in the corona criterion
Lemma~\ref{corona-lem} is also needed now.

\begin{thm}   \label{sigma-thm}
Suppose $a(x,\eta)\in \widetilde S^d_{1,1}(\Rn\times \Rn)$, i.e.\ $a(x,\eta)$
fulfils one of the equivalent conditions in Theorem~\ref{a*-thm}. Then 
the conclusions of Theorems~\ref{tdc-thm}--\ref{a2a-thm} remain valid for $a(x,D)$;
in particular $D(a(x,D))=\cal S'(\Rn)$.
\end{thm}

\begin{proof}
The continuity on $\cal S'$ is assured by Theorem~\ref{tildeS-thm}. 
For the convergence of the series in the
paradifferential splitting, it is convenient to write,
in the notation of \eqref{Hsigma-eq} ff, 
\begin{equation}
  a(x,\eta)=(a(x,\eta)-a_{\chi,1}(x,\eta))+a_{\chi,1}(x,\eta),
\end{equation}
where $a-a_{\chi,1}$ satisfies \eqref{tdc-cnd} for $B=1$, 
so that Theorem~\ref{tdc-thm} applies to it. 
As $a_{\chi,1}$ is in
$\widetilde S^d_{1,1}$ like $a$ and $a-a_{\chi,1}$ (the latter
by Proposition~\ref{GB-prop}), one may reduce to the case in which
\begin{equation}
  \hat a(x,\eta)\ne0\implies \max(1,|\xi+\eta|)\le |\eta|. 
  \label{ahat-eq}
\end{equation}

Continuing under this assumption, it is according to
Theorems~\ref{a123-thm} and \ref{a2a-thm} enough to show for all $u\in \cal S'$ that there is
convergence of the two contributions to $a^{(2)}(x,D)u$,
\begin{equation}
  \sum_{k=0}^\infty (a^k-a^{k-h})(x,D)u_k,
\qquad
\sum_{k=1}^\infty a_k(x,D)(u^{k-1}-u^{k-h}).
  \label{2a2-eq}
\end{equation}
Since the terms here are functions of polynomial growth by
Proposition~\ref{pe13-prop}, it suffices to improve the estimates there; and
to do so for $k\ge h$.

Using H{\"o}rmander's localisation to a neighbourhood of $\cal T$,
cf.\ \eqref{chi1-eq}--\eqref{chi3-eq}, one arrives at
\begin{equation}
  \hat a_{k,\chi,\varepsilon}(\xi,\eta)=
  \hat a(\xi,\eta)\varphi(2^{-k}\xi)\chi(\xi+\eta,\varepsilon\eta),
  \label{akke-id}
\end{equation}
This leaves the remainder
$b_k(x,\eta)=a_k(x,\eta)-a_{k,\chi,\varepsilon}(x,\eta)$,
that applied to the above difference
$v_k=u^{k-1}-u^{k-h}
  =\cal F^{-1}((\varphi(2^{1-k}\cdot )-\varphi(2^{h-k}\cdot))\hat u)$
gives
\begin{equation}
  a_k(x,D)v_k=a_{k,\chi,\varepsilon}(x,D)v_k+b_k(x,D)v_k .
\end{equation}
To utilise the pointwise estimates, fix $N\ge\order_{\cal S'}(\hat u)$
so that $d+N\ne0$; and pick $\Psi\in C^\infty_0(\Rn)$ equal to $1$ in a neighbourhood of the corona 
$\tfrac{r}{R}2^{-1-h}\le |\eta|\le 1$ and equal to $0$
outside the set with $\tfrac{r}{R}2^{-2-h}\le |\eta|\le 2$. 
Taking the dilated function $\Psi(\eta/(R2^k))$ as the auxiliary function  
in the symbol factor,
the factorisation inequality \eqref{Fau*-eq} and Theorem~\ref{Fa-thm} give
\begin{equation}
  |a_{k,\chi,\varepsilon}(x,D)v_k(x)|
  F_{a_{k,\chi,\varepsilon}}(N,R2^k;x)v^*_k(N,R2^k;x)
  \end{equation}
which is estimated from above by
 \begin{equation} 
  cv^*_k(x)\sum_{|\alpha|=0}^{N+[n/2]+1}
  (\int_{r2^{k-h-2}\le |\eta|\le R 2^{k+1}} |(R2^k)^{|\alpha|-n/2}
   D^\alpha_\eta a_{k,\chi,\varepsilon}(x,\eta)|^2\,d\eta)^{1/2}.
\end{equation}
Here the ratio of the limits is $2R/(r2^{-h-2})>32$, 
so with extension to $R2^{k+1-L}\le |\eta|\le R2^{k+1}$,
there is $L\ge 6$ dyadic coronas.
This gives an estimate by $c(R2^k)^d L^{1/2}
N_{\chi,\varepsilon,\alpha}(a_k)$. In addition,
Minkowski's inequality gives
\begin{equation}
\begin{split}
  N_{\chi,\varepsilon,\alpha}(a_k)
&\le 
  \sup_{\rho>0}\rho^{|\alpha|-d}\int_{\Rn} |2^{kn}\check \varphi(2^{k}y)|
   (\int_{\rho\le |\eta|\le 2\rho}    |D^\alpha_\eta 
   a_{\chi,\varepsilon}(x-y,\eta)|^2\,\frac{d\eta}{\rho^n})^{1/2}\,dy
 \\
  &\le c N_{\chi,\varepsilon,\alpha}(a).
\end{split}  
\end{equation}
So  it follows from the above that
\begin{equation}
  |a_{k,\chi,\varepsilon}(x,D)v_k(x)|\le  cv^*_k(N,R2^k;x)
  \big(\sum_{|\alpha|\le N+[n/2]+1}c_{\alpha,\sigma}
  \varepsilon^{\sigma+n/2-|\alpha|}\big)
  L^{1/2} (R2^k)^{d}.
  \label{akkev-eq}
\end{equation}
Using Lemma~\ref{u*-lem} and taking $\varepsilon=2^{-k\theta}$, 
say for $\theta=1/2$ this gives
\begin{equation}
  |a_{k,\chi,2^{-k\theta}}(x,D)v_k(x)| \le c(1+|x|)^N 2^{-k(\sigma-1-2d-3N)/2}.
  \label{akkev'-eq}
\end{equation}
Choosing  $\sigma>3N+2d+1$, the series 
$\sum_k\dual{a_{k,\chi,\varepsilon}(x,D)v_k}{\phi}$ 
converges rapidly for $\phi\in \cal S$.

To treat $\sum_{k=0}^\infty b_k(x,D)v_k$ it is observed that 
$\hat a_{k,\chi,2^{-k\theta}}(x,\eta)=\hat a_k(x,\eta)$ holds in the set where
$\chi(\xi+\eta,2^{-k\theta}\eta)=1$, 
that is, when $2\max(1,|\xi+\eta|)\le  2^{-k\theta}|\eta|$, 
so by \eqref{ahat-eq},
\begin{equation}
  \supp \hat b_k\subset \Set{(\xi,\eta)}{2^{-1-k\theta}|\eta|\le 
  \max(1,|\xi+\eta|)\le |\eta|}.
\end{equation}
This implies by Theorem~\ref{supp-thm} that $\zeta=\xi+\eta$ is in
$\supp\cal Fb_k(x,D)v_k$ only if both 
\begin{gather}
  |\zeta|\le |\eta|\le R2^k
\\
  \max(1,|\zeta|)\ge 2^{-1-k\theta}|\eta|\ge r2^{k(1-\theta)-h-2}.  
\end{gather}
When $k$ fulfils $2^{k(1-\theta)}>2^{h+2}/r$, so that the last right-hand side is $>1$,
these inequalities give
\begin{equation}
  (r2^{-h-2})2^{k(1-\theta)}\le |\zeta|\le R2^k.
  \label{bkcorona-eq}
\end{equation}
This shows that the corona condition \eqref{DAC-cnd}
in Lemma~\ref{corona-lem} is fulfilled
for $\theta_0=1-\theta=1/2$ and $\theta_1=1$, and the growth condition 
\eqref{CM-cnd} is easily checked since both 
$a_{k,\chi,\varepsilon}(x,D)v_k$ and $a_k(x,D)v_k$ 
are estimated by $2^{k(N+d_+)}(1+|x|)^{N+d_+}$, 
as can be seen from \eqref{akkev'-eq} and Proposition~\ref{cutoff-prop},
respectively. Hence $\sum b_k(x,D)v_k$ converges rapidly.

For the series $  \sum_{k=0}^\infty |\dual{(a^k-a^{k-h})(x,D)u_k}{\phi}|$
it is not complicated to modify the above. Indeed, the pointwise estimates
of the $v_k^*$ are easily carried over to $u_k^*$, for $R2^k$ was used as
the outer spectral radius of $v_k$; and $r2^{k-h-1}$ may also be used as the
inner spectral radius of $u_k$. In addition the symbol $a^k-a^{k-h}$ can be
treated by replacing $\varphi(2^{-k}\xi)$ by $\psi(2^{-k}\xi)-\psi(2^{h-k}\xi)$
in \eqref{akke-id} ff., for the use of Minkowski's inequality will now give
the factor $\int|\psi-\psi(2^h\cdot)|\,dy$
in the constant. For the remainder
\begin{equation}
  \tilde b_k(x,D)u_k=(a^k-a^{k-h})(x,D)u_k
       -(a^k-a^{k-h})_{\chi,\varepsilon}(x,D)u_k
\end{equation}
one can apply the treatment of $b_k(x,D)v_k$ verbatim.
\end{proof}

\begin{rem}
The analysis in Theorem~\ref{sigma-thm} is also exploited in the $L_p$-theory
of type $1,1$-operators in \cite{JJ11lp}. Indeed, the main ideas of the above proof was used in
\cite[Sect.~5.3]{JJ11lp} to derive
certain continuity results in the Lizorkin--Triebel scale $F^s_{p,q}(\Rn)$ for $p<1$, which 
(except for a small loss of smoothness)
generalise results of Hounie and dos Santos Kapp \cite{HoSK09}.  
\end{rem}

\section{Final Remarks}   \label{final-sect}
In view of the satisfying results on type $1,1$-operators in $\cal S'(\Rn)$, 
cf.\ Section~\ref{split-sect}, and the continuity results in the scales 
$H^s_p$, $C^s_*$, $F^{s}_{p,q}$ and $B^{s}_{p,q}$ presented in
\cite{JJ11lp},
their somewhat unusual definition by vanishing frequency modulation in
Definition~\ref{a11-defn} should be well motivated.

As an open problem, it remains to characterise the type $1,1$-operators 
$a(x,D)$ that are
everywhere defined and continuous on $\cal S'(\Rn)$. For this it was
shown above to be sufficient that $a(x,\eta)$ is in $\widetilde S^d_{1,1}(\Rn\times \Rn)$, 
and it could of course be conjectured that this is necessary as well.

Similarly, since the works of Bourdaud and H{\"o}rmander, 
cf.\ \cite[Ch.~IV]{Bou83}, \cite{Bou88}, \cite{H88,H89} and also \cite{H97}, 
it has remained an open problem to determine 
\begin{equation}
  \BBb B(L_2(\Rn))\cap \OP(S^0_{1,1}).  
  \label{BO-eq}
\end{equation} 
Indeed, this set was shown by Bourdaud to contain the self-adjoint subclass
$\OP(\widetilde{S}^0_{1,1})$, and this sufficient condition  
has led some authors to a few unfortunate statements, for example that lack
of $L_2$-boundedness for $\OP(S^0_{1,1})$ is 
``attributable to the lack of self adjointness''. 
But self-adjointness is not necessary, since already 
Bourdaud, by modification of Ching's operator \eqref{ching-eq},
gave an example \cite[p.~1069]{Bou88} of an operator 
$\sigma(x,D)$ in $ {\BBb B}(L_2)\bigcap\OP(S^0_{1,1}\setminus \widetilde S^0_{1,1})$;
that is, this $\sigma(x,D)^*$ is not of type $1,1$.

However, it could be observed that 
$N_{\chi,\varepsilon,\alpha}(a_\theta)=\cal O(\varepsilon^{n/2-|\alpha|})$  by Lemma~\ref{Heps-lem}
is valid for Ching's symbol $a_\theta$ and that this estimate is sharp for the $L_2$-unbounded version
of $a_\theta(x,D)$, by the last part of Example~\ref{Heps-exmp}. 
Therefore, the condition
\begin{equation}
  N_{\chi,\varepsilon,\alpha}(a)=o(\varepsilon^{n/2-|\alpha|})
  \quad\text{ for $\varepsilon\to0$}
\end{equation}
is conjectured to be necessary for
$L_2$-continuity of a given $a(x,D)$ in $\OP(S^0_{1,1})$.

\appendix
\section{Dyadic Corona Criteria}   \label{corona-app}
Convergence of a series $\sum_{j=0}^\infty u_j$ of
temperate distributions follows if the $u_j$ both fulfil a growth condition
and have their spectra in suitable dyadic coronas. This is a special
case of Lemma~\ref{corona-lem}, which for $\theta_0=\theta_1=1$  
was given by Coifman and Meyer \cite[Ch.~15]{MeCo97} without
arguments.

Extending the proof given in \cite{JoSi08}, 
the refined version in Lemma~\ref{corona-lem} allows the inner and outer radii of the
spectra to grow at different exponential rates $\theta_0<\theta_1$, even though the number of
overlapping spectra increases with $j$. This is crucial for Theorem~\ref{sigma-thm},
so a full proof is given.

\begin{lem}
  \label{corona-lem}
$1^\circ$~Let $(u_j)_{j\in \N_0}$ be a sequence in $\cal S'(\Rn)$ fulfilling 
that there exist $A>1$ and $\theta_1\ge\theta_0>0$ such that 
$\supp \hat u_0\subset\{\,\xi\mid |\xi|\le A\,\}$
while for $j\ge1$
   \begin{equation}
   \supp \hat u_j\subset\{\,\xi\mid 
       \tfrac{1}{A}2^{j\theta_0}\le |\xi|\le A2^{j\theta_1} \,\}, 
  \label{DAC-cnd}
   \end{equation}
and that for suitable constants $C\ge0$, $N\ge0$,
\begin{equation}
  |u_j(x)|\le C 2^{jN\theta_1}(1+|x|)^N \text{ for all $j\ge0$}.
  \label{CM-cnd}  
\end{equation}
Then $\sum_{j=0}^\infty u_j$ converges rapidly 
in $\cal S'(\Rn)$ to a distribution $u$,
for which $\hat u$ is of order $N$.

$2^\circ$~For every $u\in \cal S'(\Rn)$ both \eqref{DAC-cnd} and
\eqref{CM-cnd} are fulfilled for $\theta_0=\theta_1=1$ by the functions  
$u_0=\Phi_0(D)u$ and $u_j=\Phi(2^{-j}D)u$
when $\Phi_0,\Phi\in C^\infty_0(\Rn)$ and $0\notin\supp\Phi$. 
In particular this is the case for a 
Littlewood--Paley decomposition 
$1=\Phi_0+\sum_{j=1}^\infty \Phi(2^{-j}\xi)$.
\end{lem}

\begin{proof}
In $2^\circ$ it is clear that $\Phi$ is supported in a corona, say
$\{\,\xi\mid \tfrac{1}{A}\le |\xi|\le A\,\}$ for a large $A>0$; hence
\eqref{DAC-cnd}. 
\eqref{CM-cnd} follows from the proof of Lemma~\ref{u*-lem}.

The proof of $1^\circ$ exploits 
a well-known construction of an auxiliary function:
taking $\psi_0\in C_0^\infty (\Rn)$ depending on $|\xi|$ alone and
so that $0\le \psi_0\le 1$ with $\psi_0(\xi)=1$ for $|\xi|\le 1/(2A)$ while
$\psi_0(\xi)=0$ for $|\xi|\ge 1/A$, then
\begin{equation}
  \frac{d}{dt}\psi_0(\frac{\xi}{t})=\psi(\frac{\xi}{t})\frac{1}{t}
  \quad\text{for}\quad \psi(\xi)=-\xi\cdot \nabla\psi_0(\xi),
  \label{ddtpsi-eq}
\end{equation}
which by integration for $1\le t\le\infty $ 
gives an uncountable partition of unity 
\begin{equation}
  1=\psi_0(\xi)+\int_1^\infty \psi(\frac{\xi}{t})\,\frac{dt}{t},
  \quad \xi\in \Rn.
\end{equation}
Clearly the support of $\psi(\xi/t)$ is compact and given by
$A|\xi|\le t\le 2A|\xi|$ when $\xi$ is fixed.
For $j\ge 1$ this implies
\begin{equation}
  \hat u_j= \hat u_j\psi_0
                +\hat u_j\int_1^\infty \psi(\frac{\xi}{t})\,\frac{dt}{t}
 =\hat u_j
  \int_{2^{j\theta_0}}^{A^2 2^{j\theta_1+1}}\psi(\frac{\xi}{t})\,\frac{dt}{t}.
\end{equation}
Defining $\psi_j\in C^\infty_0(\Rn)$ as the last integral here, 
$\psi_j=1$ on $\supp \hat u_j$; so if $\varphi\in \cal S$,
\begin{equation}
  |\dual{u_j}{\overline{\varphi}}|\le 
  \Nrm{(1+|x|^2)^{-\tfrac{N+n}{2}}u_j}{2}
  \Nrm{(1+|x|^2)^{\tfrac{N+n}{2}}\cal F^{-1}(\psi_j\hat\varphi)}{2}.
\end{equation}
The first norm is $\cal O(2^{N\theta_1 j})$ by \eqref{CM-cnd}. 
For the second, note that
\begin{equation}
  \supp \psi_j\subset \set{\xi\in \Rn}{A^{-1}2^{j\theta_0-1}\le |\xi|
  \le A2^{j\theta_1+1}}
\end{equation}
and $\nrm{D^\alpha\psi_j}{\infty }\le 
   2^{-j\theta_0|\alpha|}\nrm{D^\alpha\psi}{\infty }/|\alpha|$ 
for $\alpha\ne0$ while 
$\nrm{\psi_j}{\infty }\le \op{diam}(\psi_0(\Rn))\le 1$ by \eqref{ddtpsi-eq}.
In addition the identity $(1+|x|^2)^{N+n}\cal F^{-1}=\cal F^{-1}(1-\lap)^{N+n}$ gives
for arbitrary $k>0$,
\begin{multline}
   \nrm{(1+|x|^2)^{N+n}\cal F^{-1} (\psi_j\hat\varphi)}{2}
\\
  \le\sum_{|\alpha|,|\beta|\le N+n}
   c_{\alpha,\beta} \nrm{D^\alpha\psi_j}{\infty}
   \nrm{(1+|\xi|)^{k+n/2}D^{\beta}\hat\varphi}{\infty}
   (\int_{2^{j\theta_0-1}/A}^\infty r^{-1-2k}\,dr)^{1/2}.
\end{multline}
Here $\nrm{D^\alpha \psi_j}{\infty }=\cal O(1)$,
so because of the $L_2$-norm 
the above is $\cal O(2^{-jk\theta_0})$ for every $k>0$.

Hence $\dual{u_j}{\overline{\varphi}}=
\cal O(2^{j(\theta_1N-\theta_0k)})$, so
$k>N\theta_1/\theta_0$ yields that $\sum_{j=0}^\infty \dual{u_j}{\varphi}$
converges.
\end{proof}

\begin{rem}   \label{corona-rem}
The above proof yields that the conjunction of 
\eqref{DAC-cnd} and \eqref{CM-cnd} implies 
$\dual{u_j}{\varphi}=\cal O(2^{-jN})$ for all  $N>0$; 
hence there is \emph{rapid} convergence of $u=\sum_{j=0}^\infty  u_j$ in $\cal S'$ in the
sense that 
$\dual{u-\sum_{j<k}u_j}{\varphi}=\sum_{j\ge k}\dual{u_j}{\varphi}= 
\cal O(2^{-kN})$ for $N>0$, $\varphi\in \cal S$.
\end{rem}

\section{The Spectral Support Rule}  \label{spectral-app}
To control the spectrum of $x\mapsto a(x,D)u$, i.e.\ the support of 
$\xi\mapsto \cal F a(x,D)u$, there is a simple rule which is
recalled here for the reader's convenience. 

Writing $\cal Fa(x,D)\cal F^{-1}(\hat u)$ instead, 
the question is clearly how the support of $\cal Fu$ is changed by the
conjugated operator $\cal Fa(x,D)\cal F^{-1}$.
In terms of its distribution kernel $\cal K(\xi,\eta)$,
cf.\ \eqref{KaFK-eq}, 
one should expect the spectrum of $a(x,D)u$ to be contained in 
\begin{equation}
  \Xi:=\supp\cal K\circ\supp\cal F u
  =\{\,\xi\in \Rn\mid \exists \eta\in \supp\hat u\colon (\xi,\eta)\in
  \supp\cal K \,\}. 
  \label{Xi-id} 
\end{equation}
For $\supp\cal Fu\Subset \Rn$ this was proved in \cite{JJ05DTL}; but in
general the closure $\overline{\Xi}$ should be used instead:

\begin{thm}
  \label{supp-thm}
Let $a\in  S^\infty_{1,1}(\Rn\times\Rn)$ and suppose 
$u\in D(a(x,D))$ has the property that \eqref{aPsi-eq} holds in the topology of $\cal S'(\Rn)$
for some $\psi\in C^\infty_0(\Rn)$ equal to $1$ around the origin. Then 
\begin{gather}
   \supp\cal F(a(x,D)u)\subset\overline{\Xi},
  \label{Xi-eq}
 \\
  \Xi=\bigl\{\,\xi+\eta \bigm| (\xi,\eta)\in \supp\cal F_{x\to\xi\,} a,\ 
     \eta\in \supp\cal F u \,\bigr\}.
  \label{Sigma-eq}
\end{gather}
When $u\in \cal F^{-1}\cal E'(\Rn)$ the $\cal S'$-convergence
holds automatically and $\Xi$ is closed for such $u$.
\end{thm}

The reader is referred to \cite{JJ08vfm} for the deduction of this
from the kernel formula. Note that whilst \eqref{KaFK-eq} yields \eqref{Sigma-eq}, 
it suffices for \eqref{Xi-eq} to take any $v\in C^\infty_0(\Rn)$ with support 
disjoint from $\Xi$ and verify that
\begin{equation}
  \dual{\cal Fa(x,D)\cal F^{-1}\hat u}{v}=\dual{\cal K}{v\otimes \hat u}=0.
\end{equation}
Here the middle expression makes sense 
as $\dual{(v\otimes \hat u)\cal K}{1}$, as noted in \cite{JJ08vfm}, using the
remarks to \cite[Def.~3.1.1]{H}. 
However, the first equality sign is in general not trivial to justify: the limit in Definition~\ref{a11-defn}
is decisive for this.

\begin{rem}  \label{supp-rem}
There is a simple proof of \eqref{Xi-eq} in the main case that $\hat u\in\cal E'$: 
If $a\in S^d_{1,0}$ and $v$ is as above, then
\eqref{Xi-id} yields $\op{dist}(\supp \cal K, \supp(v\otimes \hat u))>0$  
since $\supp\hat u\Subset\Rn$. 
So with $\hat u_\varepsilon=\varphi_\varepsilon*\hat u$ 
for some $\varphi\in C^\infty_0(\Rn)$ with $\hat \varphi(0)=1$,
$\varphi_\varepsilon=\varepsilon^{-n}\varphi(\cdot /\varepsilon)$,
all sufficiently small $\varepsilon>0$ give
\begin{equation}
  \supp\cal K\bigcap \supp v\otimes \hat u_\varepsilon=\emptyset.
\end{equation} 
Therefore one has, since $\hat u_\varepsilon\in C^\infty_0(\Rn)$,
\begin{equation}
  \dual{\cal Fa(x,D)\cal F^{-1}\hat u}{v}
   = \lim_{\varepsilon\to0}   \dual{\cal Fa(x,D)\cal F^{-1}\hat u_\varepsilon}{v}
   = \lim_{\varepsilon\to0}\dual{\cal K}{v\otimes \hat u_\varepsilon}=0.
  \label{FaFuv-eq}
\end{equation}
For general $b(x,\eta)$ in $S^d_{1,1}$ one may set
$a(x,\eta)=b(x,\eta)\chi(\eta)$ for a $\chi\in C^\infty_0$ equal to
$1$ on an open ball containing $\supp\hat u$. Then $a$ is in
$S^{-\infty}$ with associated kernel $\cal K_a(\xi,\eta)=\cal K_b(\xi,\eta)\chi(\eta)$ because of 
\eqref{KaFK-eq}. Moreover, the set $\Xi$ is unchanged by this replacement, so \eqref{FaFuv-eq}
gives
\begin{equation}
  \dual{\cal Fb(x,D)\cal F^{-1}\hat u}{v}=\dual{\cal Fa(x,D)\cal F^{-1}\hat u}{v}
  =\lim_{\varepsilon\to0}\dual{\cal K_a}{v\otimes\hat u_\varepsilon}=0.
\end{equation}
\end{rem}

The argument in Remark~\ref{supp-rem} clearly covers the
applications of Theorem~\ref{supp-thm} in this paper. 

%
\providecommand{\bysame}{\leavevmode\hbox to3em{\hrulefill}\thinspace}
\providecommand{\MR}{\relax\ifhmode\unskip\space\fi MR }
\providecommand{\MRhref}[2]{%
  \href{http://www.ams.org/mathscinet-getitem?mr=#1}{#2}
}
\providecommand{\href}[2]{#2}

\end{document}